\documentclass{amsart}
\usepackage{graphicx}
\usepackage{pslatex}
\usepackage{amsmath,amsfonts}
\usepackage{rotating}
\usepackage{amssymb}
\usepackage{verbatim}
\usepackage[colorlinks=true, linkcolor=black, citecolor=black]{hyperref}

\newtheorem{theorem}{Theorem}[section]
\newtheorem{lemma}[theorem]{Lemma}
\newtheorem{proposition}[theorem]{Proposition}
\newtheorem{corollary}[theorem]{Corollary}

\theoremstyle{definition}
\newtheorem{definition}[theorem]{Definition}

\theoremstyle{remark}
\newtheorem{remark}[theorem]{Remark}

\numberwithin{equation}{section}

\def\R{\mathbb R}
\def\C{\mathbb C}
\def\Z{\mathbb Z}
\def\N{\mathbb N}

\def\supp{\text{supp}}

\def\({\left(}
\def\){\right)}
\def\[{\left[}
\def\]{\right]}
\def\<{\left<}
\def\>{\right>}
\def\less{\lesssim}

\newcommand{\norm}[2]{\left|\left|#1\right|\right|_{#2}}

\begin{document}

\title{A Bilinear T(b) Theorem for Singular Integral Operators}

\author{Jarod Hart}
\address{Department of Mathematics, University of Kansas, Lawrence, Kansas 66044}
\curraddr{Department of Mathematics,
Wayne State University, 656 W. Kirby, Detroit, Michigan 48202}
\email{jarod.hart@wayne.edu}
\thanks{Research supported in part by NSF Grant \#DMS1069015.}


\date{\today}

\dedicatory{ }


\begin{abstract}
In this work, we present a bilinear Tb theorem for singular integral operators of Calder\'on-Zygmund type.  We prove some new accretive type Littlewood-Paley results and construct a bilinear paraproduct for a para-accretive function setting.   As an application of our bilinear Tb theorem, we prove product Lebesgue space bounds for bilinear Riesz transforms defined on Lipschitz curves.
\end{abstract}

\maketitle

\section{Introduction}

In the development of Calder\'on-Zygmund singular integral operator theory, measuring cancellation of operators via testing conditions has become a central theme through T1 and Tb theorems.  In the 1980's, David-Journ\'e \cite{DJ} proved the original T1 theorem, which gave a characterization of $L^2$ boundedness for Calder\'on-Zygmund operators.  Driven by the Cauchy integral operator, in the late 1980's David-Journ\'e-Semmes \cite{DJS} and McIntosh-Meyer \cite{McM} proved Tb theorems, which are also characterizations of $L^2$ bounds for Calder\'on-Zygmund operators based on perturbed testing conditions; see also \cite{Chr}.  We state the version from \cite{DJS} to compare to the bilinear version we present in this work.\\

\vspace{-.2cm}

\noindent{\bf Tb Theorem.} Let $b_0,b_1$ be para-accretive functions.  Assume that $T$ is a singular integral operator of Calder\'on-Zygmund type associated to $b_0,b_1$.  Then $T$ can be extended to a bounded operator on $L^2$ if and only if $M_{b_0}TM_{b_1}$ satisfies the weak boundedness property and $M_{b_0}T(b_1),M_{b_1}T^*(b_0)\in BMO$.\\

\vspace{-.2cm}

It should be noted here that we use slightly different notation than was used in \cite{DJS}.  We require here that $M_{b_0}T(b_1),M_{b_1}T^*(b_0)\in BMO$, whereas in \cite{DJS} this condition was written $T(b_1),T^*(b_0)\in BMO$.  These two conditions are equivalent, only using different notation.  Later we discuss why we use a different notation; see Remark \ref{r:BMOnotation}.

From the late 1980's to the early 2000's, multilinear Calder\'on-Zygmund theory was developed and { multilinear T1 theorems and boundedness results were obtained} by Christ-Journ\'e \cite{CJ} and Grafakos-Torres \cite{GT1}. { Many analogs of the linear theory have been found in the multilinear setting,} but to date there has been no multilinear Tb theorem.  In this work we prove a bilinear Tb theorem (which can be naturally extended also to a higher degree of multilinearity).  The proof presented in this work does not rely on the linear Tb theorem of David-Journ\'e-Semmes \cite{DJS}, McIntosh-Meyer \cite{McM}, or Christ \cite{Chr}.  Furthermore a new proof of the linear Tb theorem can be easily extracted from the work in this paper.  { We state our main result.}

\begin{theorem}\label{t:fullTb}
Let $b_0,b_1,b_2$ be para-accretive functions.  Assume that $T$ is a bilinear singular integral operator of Calder\'on-Zygmund type associated to $b_0,b_1,b_2$.  Then $T$ can be extended to a bounded operator from $L^{p_1}\times L^{p_2}$ into $L^p$ for all $1<p_1,p_2<\infty$ satisfying $\frac{1}{p_1}+\frac{1}{p_2}=\frac{1}{P1}$ if and only if $M_{b_0}T(M_{b_1}\,\cdot\,,M_{b_2}\,\cdot\,)$ satisfies the weak boundedness property and $M_{b_0}T(b_1,b_2),M_{b_1}T^{*1}(b_0,b_2),M_{b_2}T^{*2}(b_1,b_0)\in BMO$.
\end{theorem}

The meaning of $M_{b_0}T(b_1,b_2)\in BMO$ is not { a priori} clear here, but we define this notations in Section 2.  We also define other terminology use through out the work, such as para-accretive, singular integral operators of Calder\'on-Zygmund type, weak boundedness property, etc.

Calder\'on \cite{Ca1} proved some convergence results for a reproducing formula of the form
\begin{align*}
\int_0^\infty\phi_t*\phi_t*f\frac{dt}{t}=f,
\end{align*}
for appropriate functions $\phi_t$, which came to be known as Calder\'on's reproducing formula.  The convergence of Calder\'on's reproducing formula holds in many function space topologies; see for example Calder\'on \cite{Ca1,Ca2}, Janson-Taibleson \cite{JT}, Frazier-Jawerth-Weiss \cite{FJW}, and the references therein.  This formula has since been generalized and reformulated in many ways.  For some general formulations of this Calder\'on reproducing formula, see the work of Coifman \cite{Coi}, Nahmod \cite{N1,N2}, and Han \cite{Ha}.  We will use some of the results from \cite{Ha} in this article.  We consider discrete versions of Calder\'on's formula where we replace convolution with $\phi_t$ with certain non-convolution integral operators indexed by a discrete parameter $k\in\Z$ instead of the continuous parameter $t>0$.  We prove a criterion for extending the convergence of perturbed discrete Calder\'on reproducing formulas from $L^p$ spaces to the Hardy space $H^1$.  More precisely, we will prove:

\begin{theorem}\label{t:H1convergence}
Let $b\in L^\infty$ be para-accretive functions and $\theta_k$ be a collection of Littlewood-Paley square function kernels such that $\Theta_kb=\Theta_k^*b=0$ for all $k\in\Z$.  Also assume that 
\begin{align*}
\sum_{k\in\Z}M_b\Theta_kM_bf=bf
\end{align*}
for any $f\in C_0^\delta$ such that $bf$ has mean zero, where the convergence holds in $L^p$ for some $1<p<\infty$.  If $\phi\in C_0^\delta$ for some $0<\delta\leq1$ such that $b\phi$ has mean zero, then $b\phi\in H^1$ and
\begin{align*}
\sum_{k\in\Z}M_b\Theta_kM_b\phi=b\phi,
\end{align*}
where the convergence holds in $H^1$.
\end{theorem}

Here $C_0^\delta=C_0^\delta(\R^n)$ denotes the collection of compactly supported, $\delta$-H\"older continuous functions from $\R^n$ into $\C$.  Also we take the typical definition of the Hardy space $H^1$ with norm $||f||_{H^1}=||f||_{L^1}+\sum_{\ell=1}^n||R_\ell f||_{L^1}$, where $R_\ell$ is the $\ell^{th}$ Reisz transform in $\R^n$ for $\ell=1,...,n$, $R_\ell f=c_n\;p.v.\frac{y_\ell}{|y|^{n+1}}*f$ and $c_n$ is a dimensional constant.  Theorem \ref{t:H1convergence} tells us that anytime we have convergence of Calder\'on's reproducing formula in $L^p$ for some $p$, then it also converges in $H^1$, for appropriate operators and functions.

The need of Theorem \ref{t:H1convergence} to prove Theorem \ref{t:fullTb} comes about in a paraproduct construction used to decompose our bilinear singular integral operator $T$ (as in Theorem \ref{t:fullTb}).  To prove Theorem \ref{t:fullTb}, we follow the ideas in \cite{DJ,DJS,Hart2} to write $T=S+L_0+L_1+L_2$, where $M_{b_0}S(b_1,b_2)=M_{b_1}S^{*1}(b_0,b_2)=M_{b_2}S^{*2}(b_1,b_0)=0$ and $L_0,L_1,L_2$ are bilinear paraproducts.  We construct these paraproducts in Section \ref{s:SIO} so that they satisfy $M_{b_0}L_0(b_1,b_2)=M_{b_0}T(b_1,b_2)$ and $M_{b_1}L_0^{*1}(b_0,b_2)=M_{b_2}L_0^{*2}(b_1,b_0)=0$ in $BMO$; likewise for $L_1$ and $L_2$.  The paraproduct $L_0$ is defined in terms of a generalized Calder\'on type reproducing formula, like the ones described in Theorem \ref{t:H1convergence}.  The $H^1$ convergence given by Theorem \ref{t:H1convergence} implies $BMO$ convergence of the formula by duality when paired with appropriate elements of $H^1$, and eventually this convergence yields $M_{b_0}L_0(b_1,b_2)=M_{b_0}T(b_1,b_2)$ for this paraproduct construction.  See Section \ref{s:SIO} for more details on the construction of these paraproducts and the decomposition $T=S+L_0+L_1+L_2$.

This article is organized the following way:  In Section 2, we set notation and give a few pertinent definitions.  In Section 3, we prove a few almost orthogonality estimates for bilinear Littlewood-Paley square function kernels and operators.  In Section 4, we prove a number of convergence results in various spaces, including the Hardy space $H^1$ convergence stated in Theorem \ref{t:H1convergence}.  In Section 5, we prove an estimate closely related to bilinear Littlewood-Paley square function theory, which will serve as an estimate for truncated Calder\'on-Zygmund operators.  In Section 6, we complete the proof of Theorem \ref{t:fullTb} by proving a reduced Tb theorem and constructing a bilinear paraproduct for the para-accretive perturbed setting.  In Section 7, we apply our bilinear Tb theorem \ref{t:fullTb} to bilinear Riesz transforms defined by principle value operators along Lipschitz curves.

{\bf Acknowledgments:}  We would like to thank Rodolfo Torres and Estela Gavosto for their help in preparing his dissertation \cite{Hart3}, which contains some of the results presented here.  We also thank Atanas Stefanov for informing him of a better result for the $H^1$ estimate in Lemma \ref{l:H1functions} than what we originally obtained, and providing a proof for it.  Finally, we thank the referees and \'Arp\'ad B\'enyi for helpful suggestions that improved this article.

\section{Definitions and Preliminaries}

We first define para-accretive functions as one of several equivalent definitions provided by David-Journ\'e-Semmes \cite{DJS}.

\begin{definition}
A function $b\in L^\infty$ is para-accretive if $b^{-1}\in L^\infty$ and there is a $c_0>0$ such that for every cube $Q$, there exists a sub-cube $R\subset Q$ such that
\begin{align*}
\frac{1}{|Q|}\left|\int_Rb(x)dx\right|\geq c_0.
\end{align*}
\end{definition}
Many results involving para-accretive functions were proved by David-Journ\'e-Semmes \cite{DJS}, McIntosh-Meyer \cite{McM}, and by Han in \cite{Ha}.  We will use a number of the results from those works.

\subsection{Bilinear Singular Integrals Associated to Para-Accretive Functions}

Next we introduce the H\"older continuous spaces and para-accretive perturbed H\"older spaces.  These are the functions spaces that we use to form our initial weak continuity assumption for $T$ in Theorem \ref{t:fullTb}, similar to the linear Tb theorem in \cite{DJS}.

\begin{definition}
Define for $0<\delta\leq1$ and $f:\R^n\rightarrow\C$
\begin{align*}
||f||_\delta=\sup_{x\neq y}\frac{|f(x)-f(y)|}{|x-y|^\delta},
\end{align*}
and the space $C^\delta=C^\delta(\R^n)$ to be the collection of all functions $f:\R^n\rightarrow\C$ such that $||f||_\delta<\infty$.  Also define $C_0^\delta=C_0^\delta(\R^n)$ to be the subspace of all compactly supported functions in $C^\delta$. It follows that $||\cdot||_\delta$ is a norm on $C_0^\delta$.  Despite conventional notation, we will take $C^1$ and $C_0^1$ to be the spaces of Lipschitz continuous functions to keep our notation consistent.  Let $b$ be a para-accretive function and define $bC_0^\delta$ to be the collection of functions $bf$ such that $f\in C_0^\delta$ with norm $||bf||_{b,\delta}=||f||_\delta$.  Also let $(bC_0^\delta)'$ be the collection of all sequentially continuous linear functionals on $bC_0^\delta$, i.e. a linear functional $W:bC_0^\delta\rightarrow\C$ is in $(bC_0^\delta)'$ if and only if
\begin{align*}
\lim_{k\rightarrow\infty}||f_k-f||_\delta=0\text{ where }f_k,f\in C_0^\delta\;\;\Longrightarrow\;\;\lim_{k\rightarrow\infty}\<W,bf_k\>=\<W,bf\>,
\end{align*}
where these are both limits of complex numbers.  Given a topological space $X$, we say that an operator $T:X\rightarrow (bC_0^\delta)'$ is continuous if 
\begin{align*}
\lim_{k\rightarrow\infty}x_k=x\text{ in }X\;\;\Longrightarrow\;\;\lim_{k\rightarrow\infty}\<T(x_k),bf\>=\<T(x),bf\>\text{ for all }f\in C_0^\delta.
\end{align*}
\end{definition}

Given a bilinear operator $T:b_1C_0^\delta\times b_2C_0^\delta\rightarrow (b_0C_0^\delta)'$ for some $\delta>0$, define the transposes of $T$ for $f_1,f_2,_3\in C_0^\delta$
\begin{align*}
\<T^{1*}(b_0f_0,b_2f_2),b_1f_1\>=\<T^{*2}(b_1f_1,b_0f_0),b_1f_1\>=\<T(b_1f_1,b_2f_2),b_0f_0\>.
\end{align*}
Then the transposes of $T$ are bilinear operators acting on the following spaces:  $T^{1*}:b_0C_0^\delta\times b_2C_0^\delta\rightarrow (b_1C_0^\delta)'$ and $T^{2*}:b_1C_0^\delta\times b_0C_0^\delta\rightarrow (b_2C_0^\delta)'$.  One could more generally define the transpose $T^{1*}$ on $(b_1C_0^\delta)''\times b_1C_0^\infty$, but this is not necessary for this work.  So we restrict the first spot of $T^{1*}$ to $b_1C_0^\delta$ instead of $(b_1C_0^\delta)''$.  Likewise for $T^{2*}$.

\begin{definition}
A function $K:\R^{3n}\backslash\{(x,x,x):x\in\R^n\}\rightarrow\C$ is a standard bilinear Calder\'on-Zygmund kernel if
\begin{align*}
&|K(x,y_1,y_2)|\less\frac{1}{(|x-y_1|+|x-y_2|)^{2n}}\text{ when }|x-y_1|+|x-y_2|\neq0\\
&|K(x,y_1,y_2)-K(x',y_1,y_2)|\less\frac{|x-x'|}{(|x-y_1|+|x-y_2|)^{2n+\gamma}},\\
&\hspace{6cm}\text{ when }|x-x'|<\max(|x-y_1|,|x-y_2|)/2\\
&|K(x,y_1,y_2)-K(x,y_1',y_2)|\less\frac{|y_1-y_1'|}{(|x-y_1|+|x-y_2|)^{2n+\gamma}},\\
&\hspace{6cm}\text{ when }|y_1-y_1'|<\max(|x-y_1|,|x-y_2|)/2\\
&|K(x,y_1,y_2)-K(x,y_1,y_2')|\less\frac{|y_2-y_2'|}{(|x-y_1|+|x-y_2|)^{2n+\gamma}},\\
&\hspace{6cm}\text{ when }|y_2-y_2'|<\max(|x-y_1|,|x-y_2|)/2.
\end{align*}
Let $b_0,b_1,b_2\in L^\infty(\R^n)$ be para-accretive functions.  We say a bilinear operator $T:b_1C_0^\delta\times b_2C_0^\delta\rightarrow(b_0C_0^\delta)'$ is a bilinear singular integral operator of Calder\'on-Zygmund type associated to $b_0,b_1,b_2$, or for short a bilinear C-Z operator associated to $b_0,b_1,b_2$, if $T$ is continuous from $b_1C_0^\delta\times b_2C_0^\delta$ into $(b_0C_0^\delta)'$ for some $\delta>0$ and there exists a standard Calder\'on-Zygmund kernel $K$ such that for all $f_1,f_2,f_3\in C_0^\delta$ with disjoint support
\begin{align*}
\<T(M_{b_1}f_1,M_{b_2}f_2),M_{b_0}f_0\>&=\int_{\R^{3n}}K(x,y_1,y_2)\prod_{i=0}^2f_i(y_i)b_i(y_i)dy_i.
\end{align*}
\end{definition}

Note that this continuity assumption for $T$ from $b_1C_0^\delta\times b_2C_0^\delta$ into $(b_0C_0^\delta)'$ is equivalent to the following:  For any $f_0,f_1,f_2,g,g_k\in C_0^\delta$ such that $g_k\rightarrow g$ in $C_0^\delta$, we have
\begin{align*}
&\lim_{k\rightarrow\infty}\<T(M_{b_1}g_k,M_{b_2}f_2),M_{b_0}f_0\>=\<T(M_{b_1}g,M_{b_2}f_2),M_{b_0}f_0\>,\\
&\lim_{k\rightarrow\infty}\<T(M_{b_1}f_1,M_{b_2}g_k),M_{b_0}f_0\>=\<T(M_{b_1}f_1,M_{b_2}g),M_{b_0}f_0\>,\\
&\lim_{k\rightarrow\infty}\<T(M_{b_1}f_1,M_{b_2}f_2),M_{b_0}g_k\>=\<T(M_{b_1}f_1,M_{b_2}f_2),M_{b_0}g\>.
\end{align*}
It follows that the continuity assumptions for a bilinear singular integral operator $T$ associated to para-accretive functions $b_0,b_1,b_2$ is symmetric under transposes.  That is, $T$ is a bilinear C-Z operator associated to $b_0,b_1,b_2$ if and only if $T^{1*}$ is a bilinear C-Z operator associated to $b_1,b_0,b_2$ if and only if $T^{2*}$ is a bilinear C-Z operator associated to $b_2,b_1,b_0$.

\begin{definition}
A function $\phi\in C_0^\infty$ is a normalized bump of order $m\in\N$ if $\supp(\phi)\subset B(0,1)$ and
\begin{align*}
\sup_{|\alpha|\leq m}||\partial^\alpha\phi||_{L^\infty}\leq1.
\end{align*}
Let $b_0,b_1,b_2\in L^\infty$ be para-accretive functions, and $T$ be an bilinear C-Z operator associated to $b_0,b_1,b_2$.  We say that $M_{b_0}T(M_{b_1}\,\cdot\,,M_{b_2}\,\cdot\,)$ satisfies the weak boundedness property (written $M_{b_0}T(M_{b_1}\,\cdot\,,M_{b_2}\,\cdot\,)\in WBP$) if there exists an $m\in\N$ such that for all normalized bumps $\phi_0,\phi_1,\phi_2\in C_0^\infty$ of order $m$, $x\in\R^n$, and $R>0$
\begin{align*}
\left|\<T(M_{b_1}\phi_1^{x,R},M_{b_2}\phi_2^{x,R}),M_{b_0}\phi_0^{x,R}\>\right|\less R^n
\end{align*}
where $\phi^{x,R}(u)=\phi(\frac{u-x}{R})$.
\end{definition}

It follows by the symmetry of this definition that $M_{b_0}T(M_{b_1}\,\cdot\,,M_{b_2}\,\cdot\,)\in WBP$ if and only if $M_{b_1}T^{1*}(M_{b_0}\,\cdot\,,M_{b_2}\,\cdot\,)\in WBP$ if and only if $M_{b_2}T^{2*}(M_{b_1}\,\cdot\,,M_{b_0}\,\cdot\,)\in WBP$.  Next we define $T$ on $(b_1C^\delta\cap L^\infty)\times(b_2C^\delta\cap L^\infty)$, so that we can make sense of the testing condition $M_{b_0}T(b_1,b_2)\in BMO$ as well as the transpose conditions.  The definition we give is essentially the same as the one given by Torres \cite{T} in the linear setting and by Grafakos-Torres \cite{GT1} in the multilinear setting.  Here we use the definition from \cite{GT1} with the necessary modifications for the accretive functions $b_0,b_1,b_2$.  A benefit of this definition versus the ones (see e.g. \cite{DJ} or \cite{DJS}) is that we define $T(b_1,b_2)$ paired with any element of $b_0C_0^\delta$, not just the ones with mean zero.  Although one must still take care to note that the definition of $T$ agrees with the given definition of $T$ when paired with elements of $b_0C_0^\delta$ with mean zero.  This is all made precise in the next definition and the remarks that follow it.

\begin{definition}\label{d:Tb}
Let $b_0,b_1,b_2$ be para-accretive function, $T$ be a bilinear singular integral operator associated to $b_0,b_1,b_2$, and $f_1,f_2\in C^\delta\cap L^\infty$.  Also fix functions $\eta_R^i\in C_0^\infty$ for $R>0$, $i=1,2$ such that $\eta_R^i\equiv1$ on $B(0,R)$ and $\supp(\eta_R^i)\subset B(0,2R)$.  Then we define 
\begin{align}
&\hspace{-.5cm}\<T(b_1f_1,b_2f_2),b_0f_0\>=\lim_{R\rightarrow\infty}\<T(\eta_R^1b_1f_1,\eta_R^2b_2f_2),b_0f_0\>\notag\\
&\hspace{1cm}-\int_{\R^3}K(0,y_1,y_2)b_0(x)f_0(x)\prod_{i=1}^2f_i(y_i)(\eta_R^i(y_i)-\eta_{R_0}^i(y_i))b_i(y_i)dx\,dy_1\,dy_2,\label{Tb}
\end{align}
where $f_0\in C_0^\delta$ and $R_0>0$ is minimal such that $\supp(f_0)\subset B(0,R_0/2)$.  When $R>2R_0$, we have
\begin{align*}
\<T(\eta_R^1b_1f_1,\eta_R^2b_2f_2),b_0f_0\>=\<T(\eta_{R_0}^1b_1f_1,\eta_{R_0}^2b_2f_2),b_0f_0\>\hspace{-1cm}&\\
&\hspace{-2cm}+\<T(\eta_{R_0}^1b_1f_1,(\eta_R^2-\eta_{R_0}^2)b_2f_2),b_0f_0\>\\
&\hspace{-2cm}+\<T((\eta_R^1-\eta_{R_0}^1)b_1f_1,\eta_{R_0}^2b_2f_2),b_0f_0\>\\
&\hspace{-2cm}+\<T((\eta_R^1-\eta_{R_0}^1)b_1f_1,(\eta_R^2-\eta_{R_0}^2)b_2f_2),b_0f_0\>\\
&\hspace{-6.93cm}=\<T(\eta_{R_0}^1b_1f_1,\eta_{R_0}^2b_2f_2),b_0f_0\>\\
&\hspace{-6.2cm}+\int_{\R^{3n}}K(y_0,y_1,y_2)\eta_{R_0}^1(y_1)(\eta_R^2(y_2)-\eta_{R_0}^2(y_2))\prod_{i=0}^2b_i(y_i)f_i(y_i)dy_0\,dy_1\,dy_2\\
&\hspace{-6.2cm}+\int_{\R^{3n}}K(y_0,y_1,y_2)(\eta_R^1(y_1)-\eta_{R_0}^1(y_1))\eta_{R_0}^2(y_2)\prod_{i=0}^2b_i(y_i)f_i(y_i)dy_0\,dy_1\,dy_2\\
&\hspace{-6.2cm}+\<T((\eta_R^1-\eta_{R_0}^1)b_1f_1,(\eta_R^2-\eta_{R_0}^2)b_2f_2),b_0f_0\>\\
&\hspace{-6.93cm}=I+II+III+IV.
\end{align*}
The first term $I$ is well defined since $\eta_{R_0}^ib_if_i\in b_iC_0^\delta$ for a fixed $R_0$ (depending on $f_0$).  We check that the first integral term $II$ is absolutely convergent:  The integrand of $II$ is bounded by $||b_0||_{L^\infty}\prod_{i=1}^2||b_i||_{L^\infty}||f_i||_{L^\infty}$ times
\begin{align*}
|K(y_0,y_1,y_2)\eta_{R_0}^1(y_1)(\eta_R^2(y_2)-\eta_{R_0}^2(y_2))f_0(y_0)|&\less\frac{|\eta_{R_0}^1(y_1)(\eta_R^2(y_2)-\eta_{R_0}^2(y_2))f_0(y_0)|}{(|y_0-y_1|+|y_0-y_2|)^{2n}}\\
&\hspace{-1cm}\leq\frac{|\eta_{R_0}^1(y_1)(\eta_R^2(y_2)-\eta_{R_0}^2(y_2))f_0(y_0)|}{(|y_0-y_1|+|y_0-y_2|/2+(R_0-R_0/2)/2)^{2n}}\\
&\hspace{-1cm}\less\frac{|\eta_{R_0}^1(y_1)f_0(y_0)|}{(R_0+|y_0-y_2|)^{2n}}.
\end{align*}
This is an $L^1(\R^{3n})$ function that is independent of $R$ (as long as $R>4R_0$),
\begin{align*}
\int_{\R^{3n}}\frac{|\eta_{R_0}^1(y_1)f_0(y_0)|}{(R_0+|y_0-y_2|)^{2n}}dy_0\,dy_1\,dy_2&\less\int_{\R^{2n}}\frac{|\eta_{R_0}^1(y_1)f_0(y_0)|}{R_0^n}dy_0\,dy_1\less ||f_0||_{L^\infty}R_0^n.
\end{align*}
Since $\eta_R\rightarrow1$ pointwise, by dominated convergence the following limit exists:
\begin{align*}
&\lim_{R\rightarrow\infty}\int_{\R^{3n}}K(y_0,y_1,y_2)\eta_{R_0}^1(y_1)(\eta_R^2(y_2)-\eta_{R_0}^2(y_2))\prod_{i=0}^2b_i(y_i)f_i(y_i)dy_0\,dy_1\,dy_2\\
&\hspace{1.5cm}=\int_{\R^{3n}}K(y_0,y_1,y_2)\eta_{R_0}^1(y_1)(1-\eta_{R_0}^2(y_2))\prod_{i=0}^2b_i(y_i)f_i(y_i)dy_0\,dy_1\,dy_2.
\end{align*}
So $\lim_{R\rightarrow\infty}II$ exists.  A symmetric argument holds for $\lim_{R\rightarrow\infty}III$.  Finally, we consider $IV$ minus the integral term from \eqref{Tb}
\begin{align*}
&IV-\int_{\R^3}K(0,y_1,y_2)b_0(y_0)f_0(y_0)\prod_{i=1}^2f_i(y_i)\eta_R^i(y_i)b_i(y_i)dy_0\,dy_1\,dy_2\\
&\hspace{4cm}=\int_{\R^{3n}}(K(y_0,y_1,y_2)-K(0,y_1,y_2))b_0(y_0)f_0(y_0)\\
&\hspace{5.5cm}\times\prod_{i=1}^2(\eta_R^i(y_i)-\eta_{R_0}^i(y_i))f_i(y_i)b_i(y_i)dy_0\,dy_1\,dy_2.
\end{align*}
Again we bound the integrand by $||b_0||_{L^\infty}\prod_{i=1}^2||b_i||_{L^\infty}||f_i||_{L^\infty}$ times
\begin{align*}
|K(y_0,y_1,y_2)-K(0,y_1,y_2)|\;|f_0(y_0)|(\eta_R^1(y_1)-\eta_{R_0}^1(y_1)&\less\frac{|y_0|^\gamma|\eta_R^1(y_1)-\eta_{R_0}^1(y_1)|}{(|y_0-y_1|+|y_0-y_2|)^{2n+\gamma}}|f_0(y_0)|\\
&\hspace{-2cm}\less\frac{|y_0|^\gamma|\eta_R^1(y_1)-\eta_{R_0}^1(y_1)|}{(|y_0-y_1|/2+R_0/4+|y_0-y_2|)^{2n+\gamma}}|f_0(y_0)|\\
&\hspace{-2cm}\less\frac{R_0^\gamma|f_0(y_0)|}{(R_0+|y_0-y_1|+|y_0-y_2|)^{2n+\gamma}},
\end{align*}
which is an $L^1(\R^{3n})$ function:
\begin{align*}
\int_{\R^{3n}}\frac{R_0^\gamma|f_0(y_0)|}{(R_0+|y_0-y_1|+|y_0-y_2|)^{2n+\gamma}}dy_0\,dy_1\,dy_2&\less\int_{\R^{2n}}\frac{R_0^\gamma|f_0(y_0)|}{(R_0+|y_0-y_1|)^{n+\gamma}}dy_0\,dy_1\\
&\less\int_{\R^n}|f_0(y_0)|dy_0\less||f_0||_{L^\infty}R_0^n.
\end{align*}
Then it follows again by dominated convergence that
\begin{align*}
&\lim_{R\rightarrow\infty}\<T((\eta_R^1-\eta_{R_0}^1)b_1f_1,(\eta_R^2-\eta_{R_0}^2)b_2f_2),b_0f_0\>\\
&\hspace{1.5cm}-\int_{\R^3}K(0,y_1,y_2)b_0(y_0)f_0(y_0)\prod_{i=1}^2f_i(y_i)(\eta_R^i(y_i)-\eta_{R_0}^i(y_i))b_i(y_i)dy_0\,dy_1\,dy_2\\\
&\hspace{.5cm}=\int_{\R^{3n}}(K(x,y_1,y_2)-K(0,y_1,y_2))b_0(x)f_0(x)\prod_{i=1}^2(1-\eta_{R_0}^i(y_i))f_i(y_i)b_i(y_i)dy_1\,dy_2\,dx,
\end{align*}
which is an absolutely convergent integral.  Therefore $T(b_1f_1,b_2f_2)$ is well defined as an element of $(b_0C_0^\delta)'$ for $f_1,f_2\in C^\delta\cap L^\infty$.  Furthermore if $f_0,f_1,f_2\in C_0^\delta$ and $b_0f_0$ has mean zero, then this definition of $T$ is consistent with the a priori definition of $T$ since
\begin{align*}
&\lim_{R\rightarrow\infty}\int_{\R^3}K(0,y_1,y_2)b_0(y_0)f_0(y_0)\prod_{i=1}^2f_i(y_i)(\eta_R^i(y_i)-\eta_{R_0}^i(y_i))b_i(y_i)dy_0\,dy_1\,dy_2\\\
&\hspace{.15cm}=\(\int_{\R^3}K(0,y_1,y_2)\prod_{i=1}^2b_i(y_i)f_i(y_i)(1-\eta_{R_0}^i(y_i))dy_1\,dy_2\)\(\int_{\R^n}b_0(y_0)f_0(y_0)dy_0\)=0,
\end{align*}
since both of these integrals are absolutely convergent.  Also, when $b_0f_0$ has mean zero in this way, the definition of $\<T(b_1,b_2),b_0f_0\>$ is independent of the choice of $\eta_R^1$ and $\eta_R^2$.  We will also use the notation $M_{b_0}T(b_1,b_2)\in BMO$ or $M_{b_0}T(b_1,b_2)=\beta$ for $\beta\in BMO$ to mean that for all $f_0\in C_0^\delta$ such that $b_0f_0$ has mean zero, the following holds:
\begin{align*}
\<T(b_1,b_2),b_0f_0\>=\<\beta,b_0f_0\>.
\end{align*}
Here the left hand side makes sense since $T(b_1,b_2)$ is defined in $(b_0C_0^\delta)'$.  The right hand side also makes sense since $b_0f_0\in H^1$ for $f_0\in C_0^\delta$ where $b_0f_0$ has mean zero.  The condition $M_{b_0}T(b_1,b_2)\in BMO$ defined here is weaker than (possibly equivalent to) $T(b_1,b_2)\in BMO$ when we can make sense of $T(b_1,b_2)$ as a locally integrable function.  This is because our definition of $M_{b_0}T(b_1,b_2)\in BMO$ only requires this equality to hold when paired with a subset of the predual space of $BMO$, namely we require this to hold for $\{b_0f:f\in C_0^\delta\text{ and }b_0f\text{ has mean zero}\}\subsetneq H^1$.  It is possible that this is equivalent through some sort of density argument, but that is not of consequence here.  So we do not pursue it any further, and use the definition of $M_{b_0}T(b_1,b_2)\in BMO$ that we have provided.  Furthermore, if $T$ is bounded from $L^{p_1}\times L^{p_2}$ into $L^p$ for some $1\leq p_1,p_2,p<\infty$, then $T$ can be defined on $L^\infty\times L^\infty$ and is bounded from $L^\infty\times L^\infty$ into $BMO$.  This is result is due to Peetre \cite{P1}, Spanne \cite{Sp}, and Stein \cite{S2} in the linear setting and Grafakos-Torres \cite{GT1} in the bilinear setting.  Hence, if $T$ is bounded, then $M_{b_0}T(b_1,b_2),M_{b_1}T^{*1}(b_0,b_2),M_{b_2}T^{*2}(b_1,b_0)\in BMO$.
\end{definition}

To end this section, we make a short remark about notation.
\begin{remark}\label{r:BMOnotation}
In \cite{DJS}, they require $Tb_1\in BMO$, by which they mean that there is a $\beta\in BMO$ so that  $\<Tb_1,f\>=\<\beta,f\>$ for all $f\in b_0C_0^\delta$ where $f$ has integral zero.  This is equivalent to saying $\<Tb_1,b_0 f\>=\<\beta,b_0 f\>$ for all $f\in C_0^\delta$ where $b_0 f$ has integral zero (note that the function f plays different ``roles'' in these two expressions).  There is a small abuse in notation of \cite{DJS}.  They say that $Tb_1\in BMO$, with no mention of the para-accretive function $b_0$, but their definition does depend on $b_0$.

In the notation used here, we say (by definition) that $M_{b_0}Tb_1\in BMO$ if there is a $\beta\in BMO$ so that $\<Tb_1,b_0 f\>=\<\beta,b_0 f\>$ for all $f\in C_0^\delta$ where $b_0 f$ has integral zero.  This is equivalent to the notion of $Tb_1\in BMO$ as defined in \cite{DJS}.  We also abuse notation here in the sense that if $M_{b_0}Tb_1\in BMO$ as we defined it, then the appropriate identification of an element in $BMO$ would be $Tb_1=\beta$, not $M_{b_0}Tb_1=\beta$ as the notation suggests.  We have two reasons for using this notation.

The first is that we felt it necessary to mention the function $b_0$ in the requirement $Tb_1\in BMO$ since, as a matter of definition, it does depend on $b_0$.

The second reason is a bit more involved.  Note that we do not assert the following:  If $M_{b_0}Tb_1\in BMO$, then $Tb_1\in BMO$.  We don't make this conclusion because 1) we only deal with pairings of the form $\<Tb_1,b_0 f\>$ for $f\in C_0^\delta$ where $b_0 f$ has mean zero, and 2) we have not shown that the collection $\{b_0 f:f\in C_0^\delta\text{ and }b_0f\text{ has mean zero}\}$ is dense in $H^1$.  If this collection is in fact dense in $H^1$, then we conclude that $M_{b_0}Tb_1\in BMO$ (as we've defined it) implies $Tb_1\in BMO$.  It may be the case that this this collection is dense in $H^1$, but it is not of consequence to us in this work.  This discussion applies to the bilinear conditions $M_{b_0}T(b_1,b_2)\in BMO$ as well with the appropriate modifications.

\end{remark}

\subsection{Function, Operator, and General Notations}

Define for $N>0$, $k\in\Z$, and $x\in\R^n$
\begin{align*}
\Phi_k^N(x)=\frac{2^{kn}}{(1+2^k|x|)^N}.
\end{align*}
For $f:\R^n\rightarrow\C$, we use the notation $f_k(x)=2^{kn}f(2^kx)$.  We will say indices $0<p,p_1,p_2<\infty$ satisfy a H\"older relationship if
\begin{align}
\frac{1}{p}=\frac{1}{p_1}+\frac{1}{p_2}.\label{Holder}
\end{align}

\begin{definition}
Let $\theta_k$ be a function from $\R^{2n}$ into $\C$ for each $k\in\Z$.  We call $\{\theta_k\}_{k\in\Z}$ a collection of Littlewood-Paley square function kernels of type $LPK(A,N,\gamma)$ for $A>0$, $N>n$, and $0<\gamma\leq1$ if for all $x,y,y'\in\R^n$ and $k\in\Z$
\begin{align}
&|\theta_k(x,y)|\leq A\Phi_k^{N+\gamma}(x-y)\label{LPker1}\\
&|\theta_k(x,y)-\theta_k(x,y')|\leq A(2^k|y-y'|)^\gamma\(\Phi_k^{N+\gamma}(x-y)+\Phi_k^{N+\gamma}(x-y')\).\label{LPker2}
\end{align}
We say that $\{\theta_k\}_{k\in\Z}$ is a collection of smooth Littlewood-Paley square function kernels of type $SLPK(A,N,\gamma)$ for $A>0$, $N>n$, and $0<\gamma\leq1$ if it satisfies \eqref{LPker1}, \eqref{LPker2}, and for all $x,x',y\in\R^n$ and $k\in\Z$
\begin{align}
&|\theta_k(x,y)-\theta_k(x',y)|\leq A(2^k|x-x'|)^\gamma\prod_{i=1}^2\(\Phi_k^{N+\gamma}(x'-y)-\Phi_k^{N+\gamma}(x'-y)\).\label{LPker3}
\end{align}
If $\{\theta_k\}$ is a collection of Littlewood-Paley square function kernels of type $LPK(A,N,\gamma)$ (respectively $SLPK(A,N,\gamma)$) for some $A>0$, $N>n$, and $0<\gamma\leq1$, then write $\{\theta_k\}\in LPK$ (respectively $\{\theta_k\}\in SLPK$).  We also define for $k\in\Z$, $x\in\R^n$, and $f\in L^1+L^\infty$
\begin{align*}
\Theta_kf(x)=\int_{\R^n}\theta_k(x,y)f(y)dy.
\end{align*}
\end{definition}

\begin{definition}
Let $\theta_k$ be a functions from $\R^{3n}$ into $\C$ for each $k\in\Z$.  We call $\{\theta_k\}_{k\in\Z}$ a collection of bilinear Littlewood-Paley square function kernels of type $BLPK(A,N,\gamma)$ for $A>0$, $N>n$, and $0<\gamma\leq1$ if for all $x,y_1,y_2,y_1',y_2'\in\R^n$ and $k\in\Z$
\begin{align}
&|\theta_k(x,y_1,y_m)|\leq A\Phi_k^{N+\gamma}(x-y_1)\Phi_k^{N+\gamma}(x-y_2)\label{BLPker1}\\
&|\theta_k(x,y_1,y_2)-\theta_k(x,y_1',y_2)|\leq A(2^k|y_1-y_1'|)^\gamma\Phi_k^{N+\gamma}(x-y_2)\notag\\
&\hspace{5.5cm}\times\(\Phi_k^{N+\gamma}(x-y_1)+\Phi_k^{N+\gamma}(x-y_1')\)\label{BLPker2}\\
&|\theta_k(x,y_1,y_2)-\theta_k(x,y_1,y_2')|\leq A(2^k|y_2-y_2'|)^\gamma\Phi_k^{N+\gamma}(x-y_1)\notag\\
&\hspace{5.5cm}\times\(\Phi_k^{N+\gamma}(x-y_2)+\Phi_k^{N+\gamma}(x-y_2')\).\label{BLPker3}
\end{align}
We say that $\{\theta_k\}_{k\in\Z}$ is a collection of smooth Littlewood-Paley square function kernels of type $SBLPK(A,N,\gamma)$ for $A>0$, $N>n$, and $0<\gamma\leq1$ if it satisfies \eqref{LPker1}-\eqref{LPker3} and for all $x,x',y_1,y_2\in\R^n$ and $k\in\Z$
\begin{align}
&|\theta_k(x,y_1,y_2)-\theta_k(x',y_1,y_2)|\leq A(2^k|x-x'|)^\gamma\prod_{i=1}^2\(\Phi_k^{N+\gamma}(x-y_i)-\Phi_k^{N+\gamma}(x'-y_i)\).\label{BLPker4}
\end{align}
If $\{\theta_k\}$ is a collection of bilinear Littlewood-Paley square function kernels of type \\$BLPK(A,N,\gamma)$ (respectively of type $SBLPK(A,N,\gamma)$) for some $A>0$, $N>n$, and $0<\gamma\leq1$, then we write $\{\theta_k\}\in BLPK$ (respectively $\{\theta_k\}\in SBLPK$).  We also define for $k\in\Z$, $x\in\R^n$, and $f_1,f_2\in L^1+L^\infty$
\begin{align*}
\Theta_k(f_1,f_2)(x)=\int_{\R^{2n}}\theta_k(x,y_1,y_2)f_1(y_1)f_2(y_2)dy_1\,dy_2.
\end{align*}
\end{definition}

\begin{remark}\label{r:equivkernelcond}
Let $\theta_k$ be a function from $\R^{3n}$ to $\C$ for each $k\in\Z$.  There exists $A_1>0$, $N_1>n$, and $0<\gamma\leq1$ such that $\{\theta_k\}$ is a collection of Littlewood-Paley square function kernels of type $SBLPK(A_1,N_1,\gamma)$ if and only if there exist $A_2>0$, $N_2>n$, and $0<\gamma\leq1$ such that for all $x,y_1,y_2,y_1',y_2\in\R^n$ and $k\in\Z$
\begin{align}
&|\theta_k(x,y_1,y_2)|\leq A_2\Phi_k^{N_2}(x-y_1)\Phi_k^{N_2}(x-y_2)\label{altkerii1}\\
&|\theta_k(x,y_1,y_2)-\theta_k(x,y_1',y_2)|\leq A_22^{2nk}(2^k|y_1-y_1'|)^{\gamma}\label{altkerii2}\\
&|\theta_k(x,y_1,y_2)-\theta_k(x,y_1,y_2')|\leq A_22^{2nk}(2^k|y_2-y_2'|)^{\gamma}\label{altkerii3}\\
&|\theta_k(x,y_1,y_2)-\theta_k(x',y_1,y_2)|\leq A_22^{2nk}(2^k|x-x'|)^{\gamma}\label{altkerii4}.
\end{align}
A similar equivalence holds for smooth square function kernels of type $BLPK(A,N,\gamma)$, $LPK(A,N,\gamma)$, and $SLPK(A,N,\gamma)$ with the appropriate modifications.
\end{remark}

\begin{proof}
Assume that $\{\theta_k\}\in SBLPK(A_1,N_1,\gamma)$, and define $A_2=2A_1$, $N_2=N_1+\gamma$, and $\gamma=\gamma$.  It follows easily that \eqref{altkerii1} holds.  Also
\begin{align*}
|\theta_k(x,y_1,y_2)-\theta_k(x,y_1',y_2)|&\leq A_1(2^k|y_1-y_1'|)^{\gamma}\Phi_k^{N_1+\gamma}(x-y_2)\\
&\hspace{2cm}\times\(\Phi_k^{N_1+\gamma}(x-y_1)+\Phi_k^{N_1+\gamma}(x-y_1')\)\\
&\leq 2A_12^{2nk}(2^k|y_1-y_1'|)^{\gamma}.
\end{align*}
A similar argument holds for regularity in the $y_2$ and $x$ spots.  Then $\theta_k$ satisfies \eqref{altkerii1}-\eqref{altkerii4}.

Conversely we assume that \eqref{altkerii1}-\eqref{altkerii4} hold.  Define $\eta=\frac{N_2-n}{2(N_2+\gamma)}$, $A_1=A_2$, $N_1=N_2(1-\eta)-\eta\gamma$, and $\gamma=\eta\gamma$.  Estimate \eqref{BLPker1} easily follows since $N_1+\gamma<N_2$.  Estimate \eqref{BLPker2} also follows since
\begin{align*}
|\theta_k(x,y_1,y_2)-\theta_k(x,y_1',y_2)|&\leq A_2(2^k|y_1-y_1'|)^{\eta\gamma}\Phi_k^{N_2(1-\eta)}(x-y_2)\\
&\hspace{1.5cm}\times\(\Phi_k^{N_2(1-\eta)}(x-y_1)+\Phi_k^{N_2(1-\eta)}(x-y_1')\)\\
&\leq A_1(2^k|y_1-y_1'|)^{\gamma}\Phi_k^{N_1+\gamma}(x-y_2)\\
&\hspace{1.5cm}\times\(\Phi_k^{N_1+\gamma}(x-y_1)+\Phi_k^{N_1+\gamma}(x-y_1')\).
\end{align*}
Note that this selection satisfies
\begin{align*}
N_1=N_2-\eta(N_2+\gamma)=\frac{N_2+n}{2}>n.
\end{align*}
Then \eqref{BLPker2} holds for this choice of $A_1$, $N_1$, and $\gamma$ as well.  Estimates \eqref{BLPker3} and \eqref{BLPker4} follow with a similar argument, and hence $\{\theta_k\}$ is a collection of Littlewood-Paley square function kernel of type $BLPK(A_1,N_1,\gamma)$.  The proofs of the other equivalences are contained in the proof of this one.
\end{proof}

\begin{remark}\label{r:prodBLPK}
If $\{\lambda_k^i\}\in LPK$ (respectively $\{\lambda_k^i\}\in SLPK$)for $i=1,2$, then $\{\theta_k\}\in BLPK$ (respectively $\{\theta_k\}\in SBLPK$) where $\theta_k$ is defined, $\theta_k(x,y_1,y_2)=\lambda_k^1(x,y_1)\lambda_k^2(x,y_2)$.
\end{remark}

\section{Almost Orthogonality Estimates}

In this section, we prove some almost orthogonality estimates for kernel functions and for operators.  These type of estimates have been well-developed over the years.  The linear version of these results were implicit in many classical works by Besov \cite{B1,B2}, Taibleson \cite{T1,T2,T3}, Peetre \cite{P1,P2,P3}, Triebel \cite{Tr1,Tr2}, and Lizorkin \cite{Liz}, and they appear explicitly in the work of Frazier-Jawerth \cite{FJ}.  The bilinear versions appear in the work of Grafakos-Torres \cite{GT2}.

\subsection{Kernel Almost Orthogonality}

We first mention a well known almost orthogonality estimate for non-negative functions:  If $M,N>n$, then for all $j,k\in\Z$
\begin{align*}
\int_{\R^n}\Phi_j^M(x-u)\Phi_k^N(u-y)du\less\Phi_j^M(x-y)+\Phi_k^N(x-y).
\end{align*}
Then next result is also a result for integrals with non-negative integrands, but this one involves regularity estimates on the functions.

\begin{proposition}\label{p:smoothAO}
If $\{\theta_k\}_{k\in\Z}\in BLPK$, then for all $j,k\in\Z$, $x,y_1,y_2\in\R^n$
\begin{align*}
\int_{\R^n}&|\theta_j(x,y_1,y_2)-\theta_j(x,u,y_2)|\Phi_k^{N+\gamma}(u-y_1)du\\
&\hspace{4.5cm}\less 2^{\gamma(j-k)}\(\Phi_j^N(x-y_1)+\Phi_k^N(x-y_1)\)\Phi_j^N(x-y_2),\\
\int_{\R^n}&|\theta_j(x,y_1,y_2)-\theta_j(x,y_1,u)|\Phi_k^{N+\gamma}(u-y_2)du\\
&\hspace{4.5cm}\less 2^{\gamma(j-k)}\Phi_j^N(x-y_1)\(\Phi_j^N(x-y_2)+\Phi_k^N(x-y_2)\),
\end{align*}
and
\begin{align*}
\int_{\R^{2n}}&|\theta_j(x,y_1,y_2)-\theta_j(x,u_1,u_2)|\Phi_k^{N+\gamma}(u_1-y_1)\Phi_k^{N+\gamma}(u_2-y_2)du_1\,du_2\\
&\hspace{4.5cm}\less 2^{\gamma(j-k)}\prod_{i=1}^2\(\Phi_j^N(x-y_i)+\Phi_k^N(x-y_i)\).
\end{align*}
\end{proposition}

\begin{proof}
Since $\{\theta_k\}_{k\in\Z}$ is of type $BLPK(A,N,\gamma)$, it follows that 
\begin{align*}
\int_{\R^n}&|\theta_j(x,y_1,y_2)-\theta_j(x,u,y_2)|\Phi_k^{N+\gamma}(u-y_1)du\\
&\less\Phi_j^N(x-y_2)\int_{\R^n}(2^j|u-y_1|)^\gamma\(\Phi_j^{N+\gamma}(x-y_1)+\Phi_j^{N+\gamma}(x-u)\)\Phi_k^{N+\gamma}(u-y_1)du\\
&\leq 2^{\gamma(j-k)}\Phi_j^N(x-y_2)\int_{\R^n}\(\Phi_j^{N+\gamma}(x-y_1)+\Phi_j^{N+\gamma}(x-u)\)\Phi_k^N(u-y_1)du\\
&\less 2^{\gamma(j-k)}\(\Phi_j^N(x-y_1)+\Phi_k^N(x-y_1)\)\Phi_j^N(x-y_2).
\end{align*}
By symmetry the second estimate holds as well.  For the third estimate, we make a similar argument,
\begin{align*}
\int_{\R^{2n}}&|\theta_j(x,y_1,y_2)-\theta_j(x,u_1,u_2)|\Phi_k^{N+\gamma}(u_1-y_1)\Phi_k^{N+\gamma}(u_2-y_2)du_1\,du_2\\
&\leq\int_{\R^{2n}}|\theta_j(x,y_1,y_2)-\theta_j(x,y_1,u_2)|\Phi_k^{N+\gamma}(u_1-y_1)\Phi_k^{N+\gamma}(u_2-y_2)du_1\,du_2\\
&\hspace{.5cm}+\int_{\R^{2n}}|\theta_j(x,y_1,u_2)-\theta_j(x,u_1,u_2)|\Phi_k^{N+\gamma}(u_1-y_1)\Phi_k^{N+\gamma}(u_2-y_2)du_1\,du_2\\
&\less2^{\gamma(j-k)}\int_{\R^{2n}}\Phi_j^N(x-y_1)\(\Phi_j^N(x-y_2)+\Phi_j^N(x-u_2)\)\Phi_k^{N+\gamma}(u_1-y_1)\Phi_k^N(u_2-y_2)du_1\,du_2\\
&\hspace{.5cm}+2^{\gamma(j-k)}\int_{\R^{2n}}\(\Phi_j^N(x-y_1)+\Phi_j^N(x-u_1)\)\Phi_j^N(x-u_2)\Phi_k^N(u_1-y_1)\Phi_k^{N+\gamma}(u_2-y_2)du_1\,du_2\\
&\less2^{\gamma(j-k)}\(\Phi_j^N(x-y_1)+\Phi_k^N(x-y_1)\)\(\Phi_j^N(x-y_2)+\Phi_k^N(x-y_2)\).
\end{align*}
This completes the proof of the proposition.
\end{proof}

\subsection{Operator Almost Orthogonality Estimates}

It is well-known that if $N>n$ and $f\in L^1+L^\infty$, then $\Phi_k*|f|(x)\less\mathcal Mf(x)$ for all $k\in\Z$, where $\mathcal M$ is the Hardy-Littlewood maximal function
\begin{align*}
\mathcal Mf(x)=\sup_{x\ni B}\frac{1}{|B|}\int_B|f(y)|dy,
\end{align*}
and here the supremum is taken over all balls $B$ containing $x$.  Next we use the kernel function almost orthogonality estimates to prove pointwise estimates for some operators.

\begin{proposition}\label{p:operatorAO}
If $\{\lambda_k\},\{\theta_k\}\in LPK$ and there exists a para-accretive function $b$ such that $\Lambda_k(b)=\Theta_k(b)=0$ for all $k\in\Z$, then for all $f\in L^1+L^\infty$ and $j,k\in\Z$
\begin{align}
|\Theta_jM_b\Lambda_k^*f(x)|\less2^{-\gamma|j-k|}\mathcal Mf(x).\label{linearoperatorAO}
\end{align}
If $\{\lambda_k\}\in LPK$, $\{\theta_k\}\in SBLPK$ and there exists a para-accretive functions $b$ such that $\Lambda_k(b)=0$ and
\begin{align*}
\int_{\R^n}\theta_k(x,y_1,y_2)b(x)dx=0
\end{align*}
for all $k\in\Z$ and $y_1,y_2\in\R^n$, then for all $f_1,f_2\in L^1+L^\infty$ and $j,k\in\Z$
\begin{align}
|\Lambda_kM_b\Theta_j(f_1,f_2)(x)|&\less2^{-\gamma|j-k|}\mathcal M(\mathcal Mf_1\cdot \mathcal Mf_2)(x)\label{adjointoperatorAO}
\end{align}
Finally, if $\{\lambda_k^1\},\{\lambda_k^2\}\in LPK$, $\{\theta_k\}\in BLPK$ and there exist para-accretive functions $b_1,b_2$ and $i\in\{1,2\}$ such that $\Lambda_k^1(b_1)\cdot\Lambda_k^2(b_2)=\Theta_k(b_1,b_2)=0$ for all $k\in\Z$, then for all $f_1,f_2\in L^1+L^\infty$ and $j,k\in\Z$
\begin{align}
|\Theta_j(M_{b_1}\Lambda_k^{1\,*}f_1,M_{b_2}\Lambda_k^{2\,*}f_2)(x)|\less2^{-\gamma|j-k|}\mathcal Mf_1(x)\mathcal Mf_2(x).\label{operatorAO}
\end{align}
\end{proposition}

Here we use capital $\Lambda_k$ to be the operator defined by integration against the kernel lower case $\lambda_k$, just like $\Theta_k$ and $\theta_k$.

\begin{proof}
We first prove \eqref{linearoperatorAO}.   Using that $\Lambda_k^*(b)=0$ and Proposition \ref{p:smoothAO}
\begin{align*}
|\Theta_jM_b\Lambda_k^*f(x)|&\leq\int_{\R^n}\left|\int_{\R^n}(\theta_j(x,u)-\theta_j(x,y))b(u)\lambda_k(y,u)du\right||f(y)|dy\\
&\less\int_{\R^{2n}}|\theta_j(x,u)-\theta_j(x,y)|\Phi_k^{N+\gamma}(y-u)|f(y)|du\,dy\\
&\less2^{\gamma(j-k)}\(\Phi_j^N*|f|(x)+\Phi_k^N*|f|(x)\)\\
&\less2^{\gamma(j-k)}\mathcal Mf(x).
\end{align*}
With a symmetric argument, the same estimate holds replacing $2^{\gamma(j-k)}$ with $2^{\gamma(k-j)}$.  Therefore \eqref{linearoperatorAO} holds.  Now we prove \eqref{adjointoperatorAO}.  We first use that $\Lambda_k(b)=0$ to estimate
\begin{align*}
&|\Lambda_kM_b\Theta_j(f_1,f_2)(x)|\\
&\hspace{0cm}\leq\int_{\R^{2n}}\left|\int_{\R^n}\lambda_k(x,u)b(u)(\theta_j(u,y_1,y_2)-\theta_j(x,y_1,y_2))du\right||f_1(y_1)f_2(y_2)|dy_1\,dy_2\\
&\hspace{0cm}\less\int_{\R^{3n}}\Phi_k^{N+\gamma}(x-u)(2^j|x-u|)^\gamma(\Phi_j^{N+\gamma}(u-y_1)\Phi_j^{N+\gamma}(u-y_2)+\Phi_j^{N+\gamma}(x-y_1)\Phi_j^{N+\gamma}(x-y_2))\\
&\hspace{8cm}\times|f_1(y_1)f_2(y_2)|du\,dy_1\,dy_2\\
&\hspace{0cm}\less2^{\gamma(j-k)}\int_{\R^{3n}}\Phi_k^N(x-u)(\Phi_j^{N+\gamma}(u-y_1)\Phi_j^{N+\gamma}(u-y_2)+\Phi_j^{N+\gamma}(x-y_1)\Phi_j^{N+\gamma}(x-y_2))\\
&\hspace{8cm}\times|f_1(y_1)f_2(y_2)|du\,dy_1\,dy_2\\
&\hspace{0cm}=2^{\gamma(j-k)}\int_{\R^n}\Phi_k^N(x-u)\prod_{i=1}^2\(\Phi_j^N*|f_i|(u)+\Phi_j^N*|f_i|(x)\)du\\
&\hspace{0cm}\less2^{\gamma(j-k)}\mathcal M(\mathcal Mf_1\cdot\mathcal  Mf_2)(x).
\end{align*}
We also have
\begin{align*}
&|\Lambda_kM_b\Theta_j(f_1,f_2)(x)|\\
&\hspace{1cm}\leq\int_{\R^{2n}}\left|\int_{\R^n}\(\lambda_k(x,u)-\lambda_k(x,y_1)\)b(u)\theta_j(u,y_1,y_2)du\right||f_1(y_1)f_2(y_2)|dy_1\,dy_2\\
&\hspace{1cm}\less 2^{(k-j)\gamma}\int_{\R^{3n}}\(\Phi_k^N(x-u)+\Phi_k^N(x-y_1)\)\prod_{i=1}^2\Phi_j^N(u-y_i)|f_i(y_i)|dy_i\,du\\
&\hspace{1cm}\leq 2^{(k-j)\gamma}\int_{\R^n}\Phi_k^N(x-u)\prod_{i=1}^2\Phi_j^N*|f_i|(u)du\\
&\hspace{2cm}+2^{(k-j)\gamma}\int_{|x-y_1|\geq|x-u|/2}\Phi_k^N(x-y_1)\prod_{i=1}^2\Phi_j^N(u-y_i)|f_i(y_i)|dy_i\,du\\
&\hspace{2.5cm}+2^{(k-j)\gamma}\int_{|x-y_1|<|x-u|/2}\Phi_k^N(x-y_1)\prod_{i=1}^2\Phi_j^N(u-y_i)|f_i(y_i)|dy_i\,du\\
&\hspace{1cm}=2^{(k-j)\gamma}(I+II+III).
\end{align*}
Note that $2^{(k-j)\gamma}I\less\mathcal M(\mathcal Mf_1\cdot \mathcal Mf_2)(x)$, which is on the right hand side of \eqref{adjointoperatorAO}.  In $II$, replace $\Phi_k^N(x-y_1)$ with $\Phi_k^N((x-u)/2)$ and it follows that $II\less I$.  So $II$ is bounded appropriately as well.  The final term, $III$ is bounded by
\begin{align*}
&\int_{|x-y_1|<|x-u|/2}\Phi_k^N(x-y_1)\frac{2^{jn}|f_1(y_1)|}{(1+2^j(|x-u|-|x-y_1|))^N}\Phi_j^N*|f_2|(u)dy_1\,du\\
&\hspace{1.5cm}\less\int_{|x-y_1|<|x-u|/2}\Phi_k^N(x-y_1)\Phi_j^N(x-u)|f_1(y_1)|\Phi_j^N*|f_2|(u)dy_1\,du\\
&\hspace{1.5cm}\less\(\int_{\R^{n}}\Phi_k^N(x-y_1)|f_1(y_1)|dy_1\)\(\int_{\R^n}\Phi_j^N(x-u)\Phi_j^N*|f_2|(u)\,du\)\\
&\hspace{1.5cm}\less\Phi_k^N*|f_1|(x)\,\Phi_j^N*|f_2|(x)\\
&\hspace{1.5cm}\leq \mathcal M(\mathcal Mf_1\cdot \mathcal Mf_2)(x).
\end{align*}
This verifies that \eqref{adjointoperatorAO} holds.  For estimate \eqref{operatorAO} when $j\leq k$, we use that $\Lambda_k^1(b_1)\cdot\Lambda_k^2(b_2)=0$ and Proposition \ref{p:smoothAO}
\begin{align*}
&|\Theta_j(M_{b_1}\Lambda_k^{1\,*}f_1,M_{b_2}\Lambda_k^{2\,*}f_2)(x)|\\
&\hspace{1cm}\leq\int_{\R^{4n}}|\theta_j(x,u_1,u_2)-\theta_j(x,y_1,y_2)|\prod_{i=1}^2|b_i(u)\lambda_k^i(y_i,u_i)f_i(y_i)|dy_i\,du_i\\
&\hspace{1cm}\less2^{\gamma(j-k)}\int_{\R^{4n}}\prod_{i=1}^2\(\Phi_j^N(x-u_i)+\Phi_j^N(x-y_i)\)\Phi_k^N(u_i-y_i)|f_i(y_i)|du_i\,dy_i\\
&\hspace{1cm}\less2^{\gamma(j-k)}\prod_{i=1}^2\int_{\R^n}\(\Phi_j^N(x-y_i)+\Phi_k^N(x-y_i)\)|f_i(y_i)|dy_i\\
&\hspace{1cm}\less2^{\gamma(j-k)}\mathcal Mf_1(x)\mathcal Mf_2(x).
\end{align*}
Finally using that $\Theta_j(b_1,b_2)=0$, it follows that
\begin{align*}
&|\Theta_j(M_{b_1}\Lambda_k^{1\,*}f_1,M_{b_2}\Lambda_k^{2\,*}f_2)(x)|\\
&\hspace{.25cm}\leq\int_{\R^{4n}}|\theta_j(x,u_1,u_2)|\left|\prod_{i=1}^2\lambda_k^i(y_i,u_i)-\prod_{i=1}^m\lambda_k^i(y_i,x)\right|\prod_{i=1}^2|b_i(u)f_i(y_i)|dy_i\,du_i\\
&\hspace{.25cm}\less\int_{\R^{2n}}\(\int_{\R^{2n}}\left|\prod_{i=1}^2\lambda_k^i(y_i,u_i)-\prod_{i=1}^2\lambda_k^i(y_i,x)\right|\prod_{i=1}^2\Phi_j^{N+\gamma}(u_i-y_i)du_i\)\prod_{i=1}^m|f_i(y_i)|dy_i\\
&\hspace{.25cm}\less2^{\gamma(k-j)}\(\Phi_j^N*|f_1|(x)+\Phi_k^N*|f_1|(x)\)\(\Phi_j^N*|f_2|(x)+\Phi_k^N*|f_2|(x)\)\\
&\hspace{.25cm}\less2^{\gamma(k-j)}\mathcal Mf_1(x)\mathcal Mf_2(x).
\end{align*}
Note that we use Remark \ref{r:prodBLPK} to see that $\lambda_k^1(x,y_1)\lambda_k^2(x,y_2)$ form a collection of kernels of type $BLPK$.  Then \eqref{operatorAO} holds as well.
\end{proof}

\section{Convergence Results}

In this section, we prove convergence results for various function spaces.  Most of these results are well known, see e.g. the work of Davide-Journ\'e-Semmes \cite{DJS} or Han \cite{Ha}, but for convenience we include them here.  We also introduce a criterion for extending the convergence of some reproducing formulas in $L^p$ for to convergence in $H^1$.

\subsection{Approximation to Identities}

\begin{proposition}\label{p:apptoid}
Suppose $p_k:\R^{2n}\rightarrow\C$ for $k\in\Z$ satisfy $|p_k(x,y)|\less\Phi_k^N(x-y)$ and $N>n$, and define $P_k$
\begin{align*}
P_kf(x)=\int_{\R^n}p_k(x,y)f(y)dy
\end{align*}
for $f\in L^1+L^\infty$.  If
\begin{align*}
\int_{\R^n}p_k(x,y)dy=1
\end{align*}
for all $k\in\Z$ and $x\in\R^n$, then $P_kf\rightarrow f$ in $L^p$ as $k\rightarrow\infty$ for all $f\in L^p$ when $1\leq p<\infty$ and $P_kf\rightarrow0$ in $L^p$ as $k\rightarrow-\infty$ for all $f\in L^p\cap L^q$ for $1\leq q<p<\infty$.
\end{proposition}

\begin{proof}
For $f\in L^p$ with $1\leq p<\infty$
\begin{align*}
||P_kf-f||_{L^p}&=\(\int_{\R^n}\left|\int_{\R^n}p_k(x,y)f(y)dy-\int_{\R^n}p_k(x,y)f(x)dy\right|^pdx\)^\frac{1}{p}\\
&=\(\int_{\R^n}\left|\int_{\R^n}p_k(x,x-2^{-k}y)(f(x-2^{-k}y)-f(x))2^{-kn}dy\right|^pdx\)^\frac{1}{p}\\
&\less\int_{\R^n}\(\int_{\R^n}\Phi_0^N(y)|f(x-2^{-k}y)-f(x)|^pdx\)^\frac{1}{p}dy\\
&\less \int_{\R^n}\Phi_0^N(y)||f(\cdot-2^{-k}y)-f||_{L^p}dy.
\end{align*}
Note that $\Phi_0^N(y)||f(\cdot-2^{-k}y)-f||_{L^p}\leq2||f||_{L^p}\Phi_0^N(y)$ which is an $L^1(\R^n)$ function independent of $k$.  So by Lebesgue dominated convergence and the continuity of translation in $||\cdot||_{L^p}$,
\begin{align*}
\lim_{k\rightarrow\infty}||P_kf-f||_{L^p}\leq \int_{\R^n}\Phi_0^N(y)\lim_{k\rightarrow\infty}||f(\cdot-2^{-k}y)-f||_{L^p}dy=0.
\end{align*}
Next we compute
\begin{align*}
|P_kf(x)|\less||\Phi_k^N||_{L^{q'}}||f||_{L^q}=2^{kn/q}||\Phi_0^N||_{L^{q'}}||f||_{L^q}.
\end{align*}
So $P_kf\rightarrow0$ almost everywhere as $k\rightarrow-\infty$.  We also have
\begin{align*}
|P_kf(x)|&\less\Phi_k^N*|f|(x)\less\mathcal Mf(x),
\end{align*}
and since $f\in L^p(\R^n)$, it follows that $\mathcal Mf\in L^p(\R^n)$ as well when $1<p<\infty$.  So by dominated convergence
\begin{align*}
\lim_{k\rightarrow-\infty}||P_kf||_{L^p}^p&=\int_{\R^n}\lim_{k\rightarrow\infty}|P_kf(x)|^pdx=0.
\end{align*}
This proves the proposition.
\end{proof}

\begin{corollary}\label{c:accapptoid}
Let $b$ be a para-accretive function.  Suppose $s_k:\R^{2n}\rightarrow\C$ for $k\in\Z$ satisfy $|s_k(x,y)|\less\Phi_k^N(x-y)$ for some $N>n$, and define $S_k$
\begin{align*}
S_kf(x)=\int_{\R^n}s_k(x,y)f(y)dy
\end{align*}
for $f\in L^1+L^\infty$.  If
\begin{align*}
\int_{\R^n}s_k(x,y)b(y)dy=1
\end{align*}
for all $k\in\Z$ and $x\in\R^n$, then $S_kM_bf\rightarrow f$ and $M_bS_kf\rightarrow f$ in $L^p$ as $k\rightarrow\infty$ for all $f\in L^p(\R^n)$ when $1\leq p<\infty$.  Also $S_kM_bf\rightarrow0$ and $M_bS_kf\rightarrow0$ in $L^p$ as $k\rightarrow-\infty$ for all $f\in L^p\cap L^q$ for $1\leq q<p<\infty$.
\end{corollary}

\begin{proof}
Define $P_kf=S_kM_bf$ with kernel $p_k$.  It is obvious that $|p_k(x,y)|\less\Phi_k^N(x-y)$, and $P_k(1)=S_k(b)=1$.  So by Proposition \ref{p:apptoid}, since $f\in L^p$ it follows that $S_kM_bf=P_kf\rightarrow f$ in $L^p$ when $f\in L^p$ and $1\leq p<\infty$.  Also when $f\in L^p\cap L^q$ fo $1\leq q<p<\infty$, it follows that $S_kM_bf=P_kf\rightarrow0$ as $k\rightarrow-\infty$.  Also $M_bS_kf=M_bP_k(b^{-1}f)$, so the same convergence properties hold for $M_bS_k$.
\end{proof}

These approximation to identities perturbed by para-accretive functions are important to this work.  They have been studied in depth by David-Journ\'e-Semmes \cite{DJS} and Han \cite{Ha}, among others. 

\begin{definition}
Let $b\in L^\infty$ be a para-accretive function.  A collection of operators $\{S_k\}_{k\in\Z}$ defined by
\begin{align*}
S_kf(x)=\int_{\R^n}s_k(x,y)f(y)dy
\end{align*}
for kernel functions $s_k:\R^{2n}\rightarrow\C$ is an approximation to identity with respect to $b$ if $\{s_k\}\in SLPK$, and
\begin{align*}
&|s_k(x,y)-s_k(x',y)-s_k(x,y')+s_k(x',y')|\leq A2^{kn}(2^k|x-x'|)^\gamma(2^k|y-y'|)^\gamma\\
&\hspace{3.5cm}\times\(\Phi_k^{N+\gamma}(x-y)+\Phi_k^{N+\gamma}(x'-y)+\Phi_k^{N+\gamma}(x-y')+\Phi_k^{N+\gamma}(x'-y')\)\\
&\int_{\R^n}s_k(x,y)b(y)dy=\int_{\R^n}s_k(x,y)b(x)dx=1.
\end{align*}
We say that an approximation to identity with respect to $b$ has compactly supported kernel if $s_k(x,y)=0$ whenever $|x-y|>2^{-k}$.
\end{definition}

\begin{remark}\label{r:accapptoid}
Given a para-accretive function $b$, we define a particular approximation to the identity with respect to $b$.  Let $\varphi\in\C_0^\infty$ be radial with integral $1$ and $\supp(\varphi)\subset B(0,1/8)$.  Define $S_k^b=P_kM_{(P_kb)^{-1}}P_k$.  It follows that $S_k^b$ is an approximation to identity with respect to $b$.  Furthermore, $S_k^b$ is self-transpose and has compactly supported kernel.  It is not trivial to see that $M_{(P_kb)^{-1}}$ is well a defined operator, but it was proved in \cite{DJS} that whenever $b$ is a para-accretive function there exists $\epsilon>0$ such that $|P_kb|\geq \epsilon>0$ uniformly in $k$.  With this fact, the proof of this remark easily follows.
\end{remark}

\begin{proposition}\label{p:smoothaccconvergence}
Let $b$ be a para-accretive function, $\{S_k\}$ be the approximation to identity with respect to $b$ that has compactly supported kernel, and $\delta_0>0$.  Then $M_bS_NM_bf\rightarrow bf$ and $M_bS_{-N}M_bf\rightarrow0$ in $bC_0^\delta$ as $N\rightarrow\infty$ for all $f\in C_0^{\delta_0}$ and $0<\delta<\min(\delta_0,\gamma)$, where $\gamma$ is the smoothness parameter for $\{s_k\}\in SLPK$.  In particular these convergence results hold for the operators defined in Remark \ref{r:accapptoid}.
\end{proposition}

\begin{proof}
Let $f\in C_0^{\delta_0}$ and $0<\delta<\delta_0$.  Without loss of generality assume that $\gamma=\delta$, where $\gamma$ is the smoothness parameter of $s_k$.  We must check that $||S_NM_bf-f||_\delta\rightarrow0$ as $N\rightarrow\infty$.  So we start by estimating
\begin{align*}
&|(S_NM_bf(x)-f(x))-(S_NM_bf(y)-f(y))|\\
&\hspace{2cm}=\left|\int_{\R^n}(s_N(x,u)(f(u)-f(x))b(u)du-\int_{\R^n}(s_N(y,u)(f(u)-f(y))b(u)du\right|\\
&\hspace{2cm}\leq||b||_{L^\infty}\int_{\R^n}|F_N^x(u)-F_N^y(u)|du
\end{align*}
where $F_N^x(u)=s_N(x,u)(f(u)-f(x))$.  Consider $u\in B(y,2^{-N})$, and it follows that
\begin{align}
|F_N^x(u)-F_N^y(u)|&=|s_N(x,u)(f(u)-f(x))-s_N(y,u)(f(u)-f(y))|\notag\\
&\leq|s_N(x,u)|\;|f(y)-f(x)|+|s_N(x,u)-s_N(y,u)|\;|(f(u)-f(y))|\notag\\
&\less ||f||_{\delta_0}2^{nN}|x-y|^{\delta_0}+||f||_{\delta_0}2^{nN}(2^N|x-y|)^{\delta_0}|y-u|^{\delta_0}\label{est1}\\
&\less ||f||_{\delta_0}2^{nN}|x-y|^{\delta_0}\notag
\end{align}
With a similar argument, it follows that for $u\in B(x,2^{-N})$, $|F_N^x(u)-F_N^y(u)|\less ||f||_{\delta_0}2^{nN}|x-y|^{\delta_0}$.  Now we may also estimate $|F_N^x(u)|$ in the following way for $u\in B(x,2^{-N})$,
\begin{align}
|F_N^x(u)|&\less2^{nN}|f(u)-f(x)|\leq ||f||_{\delta_0}2^{nN}|u-x|^{\delta_0}\leq ||f||_{\delta_0}2^{nN}2^{-\delta_0N}.\label{est2}
\end{align}
Using the support properties of $s_k$, we have that $\supp(F_N^x-F_N^y)\subset B(x,2^{-N})\cup B(y,2^{-N})$.  Then it follows from \eqref{est1}, \eqref{est2}, and $\frac{\delta}{\delta_0}\in(0,1)$ that
\begin{align*}
|F_N^x(u)-F_N^y(u)|&\less\(||f||_{\delta_0}2^{nN}|x-y|^{\delta_0}\)^\frac{\delta}{\delta_0}\(||f||_{\delta_0}2^{nN}2^{-\delta_0N}\)^{1-\frac{\delta}{\delta_0}}\\
&\less ||f||_{\delta_0}2^{nN}|x-y|^\delta2^{-(\delta_0-\delta)N}.
\end{align*}
Therefore $S_NM_bf\rightarrow f$ in $||\cdot||_\delta$ since
\begin{align*}
\frac{|(S_NM_bf(x)-f(x))-(S_NM_bf(y)-f(y))|}{|x-y|^\delta}&\leq\frac{1}{|x-y|^\delta}\int_{\R^n}|F_N^x(u)-F_N^y(u)|du\\
&\less ||f||_{\delta_0}2^{-(\delta_0-\delta)N}\int_{B(x,2^{-N})\cup B(y,2^{-N})}2^{nN}du\\
&\less||f||_{\delta_0}2^{-(\delta_0-\delta)N}.
\end{align*}
This proves that $S_NM_bf\rightarrow f$ in $C_0^\delta$ as $N\rightarrow\infty$.  Now we consider $S_{-N}M_bf$ as $N\rightarrow\infty$.  We also have
\begin{align*}
\frac{|S_{-N}M_bf(x)-S_{-N}M_bf(y)|}{|x-y|^\delta}&\leq\frac{1}{|x-y|^\delta}\int_{\R^n}|s_{-N}(x,u)-s_{-N}(y,u)|\;|b(u)f(u)|du\\
&\less\frac{||f||_{L^\infty}}{|x-y|^\delta}\(\int_{|x-u|<2^N}+\int_{|y-u|<2^N}\)2^{-nN}(2^{-N}|x-y|)^\delta du\\
&\less||f||_{L^\infty}2^{-\delta N}.
\end{align*}
Note that $||f||_{L^\infty}<\infty$ since $f$ is continuous and compactly supported.  Therefore $S_NM_bf\rightarrow f$ and $S_{-N}M_bf\rightarrow0$ as $N\rightarrow\infty$ in the topology of $C_0^\delta$.
\end{proof}

\subsection{Reproducing Formulas}

We state a Calder\'on type reproducing formula for the para-accretive setting, which was constructed by Han in \cite{Ha}.

\begin{theorem}\label{t:Han}
Let $b\in L^\infty$ be a para-accretive function and $S_k^b$ for $k\in\Z$ be approximation to the identity operators with respect to $b$.  Define $D_k^b=S_{k+1}^b-S_k^b$.  There exist operators $\widetilde D_k^b$ such that
\begin{align}
\sum_{k\in\Z}\widetilde D_k^bM_bD_k^bM_bf=bf\label{hanformula}
\end{align}
in $L^p$ for all $1<p<\infty$ and $f\in C_0^\delta$ such that $bf$ has mean zero.  Furthermore, $\widetilde D_k^b(b)=\widetilde D_k^{b\,*}(b)=0$ and $\widetilde D_k^b$ is defined by
\begin{align*}
\widetilde D_k^bf(x)=\int_{\R^n}\widetilde d_k^b(x,y)f(y)dy
\end{align*}
where $\{\widetilde d_k^{b\,*}\}\in LPK$, where $\widetilde d_k^{b\,*}(x,y)=\widetilde d_k^b(y,x)$ are the kernels associated with $\widetilde D_k^{b\,*}$
\end{theorem}

We will use this formula extensively, and in fact, we need this formula in $H^1$ as well to construct the accretive type para-product in Section \ref{s:SIO}.  We will prove that this reproducing formula holds in $H^1$ in Theorem \ref{t:H1convergence} and its Corollary \ref{c:accH1}.  First we prove a lemma.  

\begin{lemma}\label{l:H1functions}
If $f:\R^n\rightarrow\C$ has mean zero and 
\begin{align*}
|f(x)|\less \Phi_j^N(x)+\Phi_k^N(x)
\end{align*}
for some $N>n$ and $j,k\in\Z$, then $f\in H^1$ and $||f||_{H^1}\less 1+|j-k|$, where the suppressed constant is independent of $j$ and $k$.
\end{lemma}

This is an extension of a result of Uchiyama \cite{U}, which is Lemma \ref{l:H1functions} when $j=k$.  Initially in \cite{Hart3}, we obtained a quadratic bound, $|j-k|^2$, for Lemma \ref{l:H1functions} using an argument involving atomic decompositions in $H^1$. Such a result suffices for our purposes, but thanks to suggestions from Atanas Stefanov we are able to obtain the linear bound stated here.   We present Stefanov's proof, which appears more natural.

\begin{proof}
The conclusion of Lemma \ref{l:H1functions} is well known for $j=k$, see e.g. the work of Uchiyama \cite{U} or Wilson \cite{W2}.  So without loss of generality we take $j\neq k$, and furthermore we suppose that $j<k$.  It is easy to see that 
\begin{align*}
||f||_{L^1}\less ||\Phi_j^N||_{L^1}+||\Phi_k^N||_{L^1}\less 1,
\end{align*}
so we may reduce the problem to proving that $||R_\ell f||_{L^1}\less k-j$ for $\ell=1,...,n$.  The strategy here is to split the norm $||R_\ell f||_{L^1}$ into two sets, where $|x|\leq2^{-j}$ and where $|x|>2^{-j}$.  We will control the first by $k-j$ and the second by $1$.  Define $p=1+\frac{1}{k-j}>1$, and use that $||R_\ell||_{L^p\rightarrow L^p}\less p'$ to estimate
\begin{align}
||\chi_{|x|\leq 2^{-j}}R_\ell f||_{L^1}&\leq||\chi_{|x|\leq 2^{-j}}||_{L^{p'}}||R_\ell f||_{L^p}\notag\\
&\less 2^{-nj/p'}p'||f||_{L^{P1}}\notag\\
&\less (k-j)2^{-nj/p'}\(2^{nj/p'}+2^{nk/p'}\)\less k-j.\label{x<1}
\end{align}
Note that here we use that $p'=k-j+1$ and hence $2^{n(k-j)/p'}\leq2^n$.  Now it remains to control
\begin{align}
||\chi_{|x|>2^{-j}}R_\ell f||_{L^1}&\leq\sum_{m=-j}^\infty||\chi_{2^{m}<|x|\leq2^{m+1}}R_\ell f||_{L^1}\notag\\
&\leq\sum_{m=-j}^\infty||\chi_{2^{m}<|x|\leq2^{m+1}}R_\ell (f\chi_{|y|\leq2^{m-1}})||_{L^1}\notag\\
&\hspace{1cm}+\sum_{m=-j}^\infty||\chi_{2^{m}<|x|\leq2^{m+1}}R_\ell (f\chi_{|y|>2^{m-1}})||_{L^1}=I+II.\label{x>1}
\end{align}
In order to estimate $I$ from \eqref{x>1}, we bound the terms of the sum by first breaking them into two pieces using the mean zero hypothesis on $f$:
\begin{align}
||\chi_{2^{m}<|x|\leq2^{m+1}}R_\ell (f\chi_{|y|\leq2^{m-1}})||_{L^1}&\hspace{0cm}=\int_{2^{m}<|x|\leq2^{m+1}}\left|R_\ell (f\chi_{|y|\leq2^{m-1}})(x)-\int_\R\frac{x_\ell}{|x|^{n+1}}f(y)dy\right|dx\notag\\
&\hspace{0cm}\leq\int_{|x|>2^{m}}\int_{|y|\leq2^{m-1}}\left|\frac{x_\ell-y_\ell}{|x-y|^{n+1}}-\frac{x_\ell}{|x|^{n+1}}\right||f(y)|dy\,dx\notag\\
&\hspace{1.5cm}+\int_{2^{m}<|x|\leq 2^{m+1}}\int_{|y|>2^{m-1}}\frac{|f(y)|}{|x|^n}dy\,dx=I_a+I_b.\label{twoterms}
\end{align}
Let $\delta=\min(1,(N-n)/2)$ and $N'=N-\delta>n$.  Then the first term of \eqref{twoterms} is bounded by
\begin{align*}
I_a\leq\int_{|x|>2^{m}}\int_{|y|\leq2^{m-1}}\frac{|y|}{|x|^{n+1}}|f(y)|dy\,dx&\leq\int_{|x|>2^{m}}\int_{|y|\leq2^{m-1}}\frac{|y|^\delta}{|x|^{n+\delta}}|f(y)|dy\,dx\\
&\hspace{0cm}\less 2^{-m\delta}\int_{\R}|y|^\delta\(\Phi_j^N(y)+\Phi_k^N(y)\)dy\\
&\hspace{0cm}\leq 2^{-m\delta}\int_{\R}\(2^{-j\delta}\Phi_j^{N'}(y)+2^{-k\delta}\Phi_k^{N'}(y)\)dy\\
&\hspace{0cm}\less 2^{-(j+m)\delta}.
\end{align*}
Note that we absorb the $2^{-k\delta}$ term into the $2^{-j\delta}$ term since $k>j$.  The second term of \eqref{twoterms} is bounded by
\begin{align*}
I_b\leq\int_{2^{m}<|x|\leq 2^{m+1}}\int_{|y|>2^{m-1}}\frac{1}{|x|^n}|f(y)|dy\,dx&\leq2^{-mn}\int_{2^{m}<|x|\leq 2^{m+1}}\int_{|y|>2^{m-1}}|f(y)|dy\,dx\\
&\leq \int_{|y|>2^{m-1}}\(\frac{2^{-j(N-n)}}{|y|^N}+\frac{2^{-k(N-n)}}{|y|^N}\)dy\\
&\less 2^{-(j+m)(N-n)}+2^{-(k+m)(N-n)}\\
&\less 2^{-(j+m)(N-n)}.
\end{align*}
Again we use that $2^{-k(N-n)}\leq2^{-j(N-n)}$ since $k>j$ and $N>n$.  Now in order to estimate $II$ from \eqref{x>1}, we bound the terms of the sum using an $L^2$ bound for $R_\ell$
\begin{align*}
||\chi_{2^{m}<|x|\leq2^{m+1}}R_\ell(f\chi_{|y|>2^{m-1}})||_{L^1}&\leq||\chi_{2^{m}<|x|\leq2^{m+1}}||_{L^2}||R_\ell (f\chi_{|y|>2^{m-1}})||_{L^2}\\
&\hspace{-2cm}\less 2^{mn/2}\(\int_{|y|>2^{m-1}}\(\Phi_j^N(y)+\Phi_k^N(y)\)^2dy\)^\frac{1}{2}\\
&\hspace{-2cm}\leq 2^{mn/2}\(\int_{|y|>2^{m-1}}\[\frac{2^{2j(n-N)}}{|y|^{2N}}+\frac{2^{2k(n-N)}}{|y|^{2N}}\]dy\)^\frac{1}{2}\\
&\hspace{-2cm}\less 2^{mn/2}\(2^{-j(N-n)}+2^{-k(N-n)}\)\(\int_{|y|>2^{m-1}}\frac{1}{|y|^{2N}}dy\)^\frac{1}{2}\\
&\hspace{-2cm}\less 2^{-(j+m)(N-n)}.
\end{align*}
Using these estimates, it follows that \eqref{x>1} is bounded in the following way:
\begin{align*}
I+II&\less \sum_{m=-j}^\infty2^{-(j+m)\delta}+ \sum_{m=-j}^\infty2^{-(j+m)(N-n)}\less 1.
\end{align*}
Therefore using \eqref{x<1} and \eqref{x>1}, it follows that $||R_\ell f||_{L^1}\less k-j$ for $\ell=1,...,n$ and hence $||f||_{H^1}\less k-j$.
\end{proof}

Now we prove Theorem \ref{t:H1convergence}.

\begin{proof}
Define for $k\in\Z$, $f_k(x)=M_b\Theta_kM_b\phi$.  It easily follows that
\begin{align*}
\int_{\R^n}f_k(x)dx&=\int_{\R^n}M_b\phi(x)\Theta_k^*b(x)dx=0.
\end{align*}
Let $R$ be large enough so that $\supp(\phi)\subset B(0,R)$.  We estimate
\begin{align*}
|f_k(x)|&\leq||b||_{L^\infty}\left|\int_{\R^n}(\theta_k(x,y)-\theta_k(x,0))b(y)\phi(y)dy\right|\\
&\less\int_{\R^n}(2^k|y|)^\gamma\(\Phi_k^N(x-y)+\Phi_k^N(x)\)|\phi(y)|dy\\
&\less 2^{\gamma k} R^\gamma\(\Phi_k^N*\Phi_0^{N}(x)+\Phi_k^N(x)\)\\
&\less 2^{\gamma k}\(\Phi_0^{N}(x)+\Phi_k^{N}(x)\).
\end{align*}
We also estimate
\begin{align*}
|f_k(x)|&\leq||b||_{L^\infty}\left|\int_{\R^n}\theta_k(x,y)b(y)(\phi(y)-\phi(x))dy\right|\\
&\less\int_{\R^n}\Phi_k^{N+\gamma}(x-y)|x-y|^\gamma(\Phi_0^N(y)+\Phi_0^N(x))dy\\
&\less2^{-\gamma k}\int_{\R^n}\Phi_k^N(x-y)(\Phi_0^N(y)+\Phi_0^N(x))dy\\
&\less2^{-\gamma k}\(\Phi_0^{N}(x)+\Phi_k^{N}(x)\).
\end{align*}
So we have proved that $|f(x)|\less2^{-\gamma|k|}(\Phi_0^{N}(x)+\Phi_k^{N}(x))$.  It follows from Lemma \ref{l:H1functions} applied with $j=0$ that 
\begin{align*}
||f_k||_{H^1}\less (1+ |k|)2^{-|k|\gamma}.
\end{align*}
Therefore
\begin{align*}
\norm{\sum_{|k|<M}f_k}{H^1}&\leq\sum_{|k|<M}\norm{f_k}{H^1}\less\sum_{k\in\Z}(1+ |k|)2^{-|k|\gamma}<\infty.
\end{align*}
Hence $\sum_{|k|<M}f_k$ is a Cauchy sequence in $H^1$, and there exists $\widetilde\phi\in H^1$ such that
\begin{align*}
\widetilde\phi=\sum_{k\in\Z}f_k=\sum_{k\in\Z}M_b\Theta_kM_b\phi.
\end{align*}
But since the reproducing formula holds for $b\phi$ in $L^p$ for some $1<p<\infty$, it follows that $\widetilde\phi=b\phi$ and the reproducing formula holds for $b\phi$ in $H^1$, which completes the proof.
\end{proof}

\begin{corollary}\label{c:accH1}
Let $b\in L^\infty$ be a para-accretive function, $S_k^b$, $D_k^b$, and $\widetilde D_k^b$ be approximation to identity and reproducing formula operator with respect to $b$ as in Theorem \ref{t:Han}.  Then for all $\delta>0$ and $\phi\in C_0^\delta$ such that $b\phi$ has mean zero,
\begin{align*}
\sum_{k\in\Z}M_b\widetilde D_kM_bD_kM_b\phi=\sum_{k\in\Z}M_bD_kM_b\phi=b\phi
\end{align*}
in $H^1$.
\end{corollary}

\begin{proof}
By Theorem \ref{t:Han}, it follows that the kernels of $\widetilde D_kM_bD_k$ and $D_k$ are Littlewood-Paley square function kernels of type $LPK$, that 
\begin{align*}
\widetilde D_kM_bD_k(b)=(\widetilde D_kM_bD_k)^*(b)=D_k(b)=D_k^*(b)=0,
\end{align*}
and finally that 
\begin{align*}
\sum_{k\in\Z}M_b\widetilde D_kM_bD_kM_b\phi=\sum_{k\in\Z}M_bD_kM_bf=b\phi
\end{align*}
in $L^p$ for all $1<p<\infty$ when $\phi\in C_0^\delta$ when $b\phi$ has mean zero.  Therefore it follows from Theorem \ref{t:H1convergence} that the formula holds in $H^1$ as well.
\end{proof}

\section{A Square Function-Like Estimate}

In this section, we work with Littlewood-Paley type square function kernel adapted to para-accretive functions, but we do not actually prove any square function bounds.  Instead we prove an estimate for a sort of ``dual pairing'' that will be useful to approximate Lebesgue space norms for the singular integral operators in the next section.

\begin{theorem}\label{t:mdualbound}
If $\{\theta_k\}\in SLPK$ and there exist para-accretive functions $b_0,b_1,b_2$ such that
\begin{align*}
\int_{\R^n}\theta_k(x,y_1,y_2)b_0(x)dx=\int_{\R^{2n}}\theta_k(x,y_1,y_2)b_1(y_1)b_2(y_2)dy_1\,dy_2=0
\end{align*}
for all $x,y_1,y_2\in\R^n$ and $k\in\Z$, then for all $1<p,p_1,p_2<\infty$ satisfying \eqref{Holder}, $f_i\in L^{p_i}$ for $i=0,1,2$ where $p_0=p'$
\begin{align*}
\sum_{k\in\Z}\left|\int_{\R^n}\Theta_k(f_1,f_2)(x)f_0(x)dx\right|\less\prod_{i=0}^2||f_i||_{L^{p_i}}
\end{align*}
\end{theorem}

\begin{proof}
Since $b_i,b_i^{-1}\in L^\infty$, it is sufficient to prove this estimate for $b_if_i$ in place of $f_i$ for $i=0,1,2$.  Fix $1<p,p_1,p_2<\infty$ satisfying \eqref{Holder}, $f_i\in C_0^\delta$ for $i=0,1,2$ and some $\delta$ where $b_if_i$ has mean zero for $i=0,1,2$.  Define 
\begin{align*}
&\Pi_j^1(f_1,f_2)(x)=M_{b_1}D_k^{b_1}M_{b_1}f_1(x)M_{b_2}S_{k+1}^{b_2}M_{b_2}f_2(x)\\
&\Pi_j^2(f_1,f_2)(x)=M_{b_1}S_k^{b_1}M_{b_1}f_1(x)M_{b_2}D_k^{b_2}M_{b_2}f_2(x),
\end{align*}
where $S_k^{b_i}$ and $D_k^{b_i}$ are defined as in Theorem \ref{t:Han}.  Then it follows that
\begin{align*}
&\Theta_k(b_1f_1,b_2f_2)\\
&\hspace{.5cm}=\lim_{N\rightarrow\infty}\Theta_k(M_{b_1}S_N^{b_1}M_{b_1}f_1,M_{b_2}S_N^{b_2}M_{b_2}f_2)-\Theta_k(M_{b_1}S_{-N}^{b_1}M_{b_1}f_1,M_{b_2}S_{-N}^{b_2}M_{b_2}f_2)\\
&\hspace{.5cm}=\lim_{N\rightarrow\infty}\sum_{j=-N}^{N-1}\Theta_k(M_{b_1}S_{k+1}^{b_1}M_{b_1}f_1,M_{b_2}S_{k+1}^{b_2}M_{b_2}f_2)-\Theta_k(M_{b_1}S_k^{b_1}M_{b_1}f_1,M_{b_2}S_k^{b_2}M_{b_2}f_2)\\
&\hspace{.5cm}=\lim_{N\rightarrow\infty}\sum_{j=-N}^{N-1}\Theta_k\Pi_j^1(f_1,f_2)+\Theta_k\Pi_j^2(f_1,f_2)
\end{align*}
where the convergence holds in $L^p$.  Then we approximate the above dual pairing in the following way
\begin{align*}
\left|\sum_{k\in\Z}\int_{\R^n}\Theta_k(b_1f_1,b_2f_2)(x)b_0(x)f_0(x)dx\right|&\leq\sum_{j,k\in\Z}\left|\int_{\R^n}\Theta_k\Pi_j^1(f_1,f_2)(x)b_0(x)f_0(x)dx\right|\\
&\hspace{.5cm}+\sum_{j,k\in\Z}\left|\int_{\R^n}\Theta_k\Pi_j^2(f_1,f_2)(x)b_0(x)f_0(x)dx\right|.
\end{align*}
These two terms are symmetric, so we only bound the first one.  The bound for the other term follows with a similar argument.  By the convergence in Theorem \ref{t:Han}, we have that
\begin{align*}
&\sum_{j,k\in\Z}\left|\int_{\R^n}\Theta_k\Pi_j^1(f_1,f_2)(x)b_0(x)f_0(x)dx\right|\\
&\hspace{1cm}\leq\sum_{j,k,\ell\in\Z}\left|\int_{\R^n}\Theta_k\Pi_j^1(\widetilde D_\ell^{b_1}M_{b_1}D_\ell^{b_1}M_{b_1}f_1,f_2)(x)b_0(x)f_0(x)dx\right|\\
&\hspace{1cm}\leq\sum_{j,k,\ell,m\in\Z}\left|\int_{\R^n}\widetilde D_m^{b_0}M_{b_0}D_m^{b_0}M_{b_0}\Theta_k\Pi_j^1(\widetilde D_\ell^{b_1}M_{b_1}D_\ell^{b_1}M_{b_1}f_1,f_2)(x)b_0(x)f_0(x)dx\right|\\
&\hspace{1cm}\leq\sum_{j,k,\ell,m\in\Z}\int_{\R^n}|D_m^{b_0}M_{b_0}\Theta_k\Pi_j^1(\widetilde D_\ell^{b_1}M_{b_1}D_\ell^{b_1}M_{b_1}f_1,f_2)(x)M_{b_0}\widetilde D_m^{b_0\,*}M_{b_0}f_0(x)|dx.
\end{align*}
By Proposition \ref{p:operatorAO} we also have the following three estimates
\begin{align*}
|D_m^{b_0}M_{b_0}\Theta_k\Pi_j^i(\widetilde D_\ell^{b_1}M_{b_1}D_\ell^{b_1}M_{b_1}f_1,f_2)(x)|&\less2^{-\gamma|m-k|}\mathcal M\(\Pi_j^1(\widetilde D_\ell^{b_1}M_{b_1}D_\ell^{b_1}M_{b_1}f_1,f_2)\)(x)\\
&\less2^{-\gamma|m-k|}\mathcal M^2\(\mathcal M(D_\ell^{b_1} M_{b_1}f_1)\cdot\mathcal  Mf_2\)(x).\\
|D_m^{b_0}M_{b_0}\Theta_k\Pi_j^1(\widetilde D_\ell^{b_1}M_{b_1}D_\ell^{b_1}M_{b_1}f_1,f_2)(x)|&\less \mathcal M(\Theta_k\Pi_j^1(\widetilde D_\ell^{b_1}M_{b_1}D_\ell^{b_1}M_{b_1}f_1,f_2))(x)\\
&\less 2^{-\gamma|k-j|}\mathcal M^2(\mathcal M(D_\ell^{b_1}M_{b_1}f_1)\cdot\mathcal Mf_2))(x)\\
|D_m^{b_0}M_{b_0}\Theta_k\Pi_j^i(\widetilde D_\ell^{b_1}M_{b_1}D_\ell^{b_1}M_{b_1}f_1,f_2)(x)|&\less \mathcal M^2(\Pi_j^1(\widetilde D_\ell^{b_1}M_{b_1}D_\ell^{b_1}M_{b_1}f_1,f_2))(x)\\
&\less 2^{-\gamma|j-\ell|}\mathcal M^2(\mathcal M(D_\ell^{b_1}M_{b_1}f_1)\cdot\mathcal Mf_2)(x)
\end{align*}
Taking the geometric mean of these three estimates, we have the following pointwise bound
\begin{align*}
&|D_m^{b_0}M_{b_0}\Theta_k\Pi_j^i(M_{b_1}\widetilde D_\ell^{b_1}M_{b_1}D_\ell^{b_1}f_1,f_2)(x)|\\
&\hspace{3cm}\less2^{-\gamma\(\frac{|m-k|}{3}+\frac{|k-j|}{3}+\frac{|j-\ell|}{3}\)}\mathcal M^2\(\mathcal M(D_\ell^{b_1}M_{b_1}f_1)\cdot\mathcal  Mf_2\)(x).
\end{align*}
Therefore
\begin{align*}
&\sum_{j,k,\ell,m\in\Z}\int_{\R^n}|D_m^{b_0}M_{b_0}\Theta_k\Pi_j^1(\widetilde D_\ell^{b_1}M_{b_1}D_\ell^{b_1}M_{b_1}f_1,f_2)(x)\widetilde D_m^{b_0\,*}M_{b_0}f_0(x)|dx\\
&\hspace{.25cm}\less\int_{\R^n}\sum_{j,k,\ell,m\in\Z}2^{-\gamma\(\frac{|m-k|}{3}+\frac{|k-j|}{3}+\frac{|j-\ell|}{3}\)} \mathcal M^2\( \mathcal M (D_\ell^{b_1} M_{b_1}f_1)\cdot \mathcal Mf_2\)(x)|\widetilde D_m^{b_0\,*}M_{b_0}f_0(x)|dx\\
&\hspace{.25cm}\leq\left|\left|\(\sum_{j,k,\ell,m\in\Z}2^{-\gamma\(\frac{|m-k|}{3}+\frac{|k-j|}{3}+\frac{|j-\ell|}{3}\)} \mathcal M^2\( \mathcal M (D_\ell^{b_1} M_{b_1}f_1)\cdot \mathcal Mf_2\)^2\)^\frac{1}{2} \right|\right|_{L^p}\\
&\hspace{4.5cm}\times\norm{\(\sum_{j,k,\ell,m\in\Z}2^{-\gamma\(\frac{|m-k|}{3}+\frac{|k-j|}{3}+\frac{|j-\ell|}{3}\)}|\widetilde D_m^{b_0\,*}M_{b_0}f_0|^2\)^\frac{1}{2}}{L^{p'}}\\
&\hspace{.25cm}\less\left|\left|\(\sum_{\ell\in\Z} \mathcal M^2\(\mathcal  M (D_\ell^{b_1}M_{b_1} f_1)\cdot \mathcal Mf_2\)^2\)^\frac{1}{2} \right|\right|_{L^p}\norm{\(\sum_{m\in\Z}|\widetilde D_m^{b_0\,*}M_{b_0}f_0|^2\)^\frac{1}{2}}{L^{p'}}\\
&\hspace{.25cm}\less\left|\left|\(\sum_{\ell\in\Z}( \mathcal M (D_\ell^{b_1}M_{b_1} f_1) \mathcal Mf_2)^2\)^\frac{1}{2} \right|\right|_{L^p}||f_0||_{L^{p'}}\\
&\hspace{.25cm}\leq\left|\left|\(\sum_{\ell\in\Z}\( \mathcal M (D_\ell^{b_1} M_{b_1}f_1)\)^2\)^\frac{1}{2} \right|\right|_{L^{p_1}}|| \mathcal Mf_2||_{L^{p_2}}||f_0||_{L^{p'}}\less\prod_{i=0}^2||f_i||_{L^{p_i}}.
\end{align*}
In the last three lines, we apply the Fefferman-Stein vector valued maximal inequality \cite{FS1}, H\"older's inequality, and the square function bounds for $D_\ell^{b_1}$ and $\widetilde D_m^{b_0\,*}$ proved by David-Journ\'ed-Semmes in \cite{DJS}.  By symmetry and density, this completes the proof.
\end{proof}

\section{Singular Integral Operators}\label{s:SIO}

In this section, we prove a reduced Tb theorem, construct a para-accretive paraproduct, and prove a full Tb theorem all in the bilinear setting.  First, we prove a few technical lemmas that relate the work in the preceding sections to singular integral operators.

\subsection{Two Technical Lemmas}

\begin{lemma}\label{l:WBP}
Let $b_0,b_1,b_2\in L^\infty$ be para-accretive functions, and assume that $T$ is a bilinear C-Z operator associated to $b_0,b_1,b_2$ such that $M_{b_0}T(M_{b_1}\,\cdot\,,M_{b_2}\,\cdot\,)\in WBP$ for normalized bumps of order $m$.  Then for all normalized bumps $\phi_0,\phi_1,\phi_2$, $R>0$ of order $m$, and $y_0,y_1,y_2\in\R^n$ such that $|y_0-y_i|\leq tR$ 
\begin{align*}
\left|\<T(M_{b_1}\phi_1^{y_1,R},M_{b_2}\phi_2^{y_2,R}),M_{b_0}\phi_0^{y_0,R}\>\right|\less (1+t)^{n+3m}R^n.
\end{align*}
\end{lemma}

\begin{proof}
Let $y_0,y_1,y_2\in\R^n$, $R>0$, and define $D=1+2t$.  Then it follows that
\begin{align*}
\left|\<T(M_{b_1}\phi_1^{y_1,R},M_{b_2}\phi_2^{y_2,R}),M_{b_0}\phi_0^{y_0,R}\>\right|&=\left|\<T(M_{b_1}\widetilde\phi_1^{y_0,DR},M_{b_2}\widetilde\phi_2^{y_0,DR}),M_{b_0}\widetilde\phi_0^{y_0,DR}\>\right|.
\end{align*}
where $\widetilde\phi_0(u)=\phi_0(Du)$ and $\widetilde\phi_i(u)=\phi_i\(Du+\frac{y_0-y_1}{R}\)$ for $i=1,2$.  If $|u|>1$, then clearly $D|u|>1$, and
\begin{align*}
\left|Du+\frac{y_0-y_1}{R}\right|\geq D|u|-\frac{|y_0-y_1|}{R}\geq(1+2t) |u|-t\geq1.
\end{align*}
So we have that $\supp(\widetilde\phi_i)\subset B(0,1)$.  It follows that $D^{-m}\widetilde\phi_i\in C_0^\infty$ are normalized bumps of order $m$, and it follows that
\begin{align*}
\left|\<T(M_{b_1}\widetilde\phi_1^{y_0,DR},M_{b_2}\widetilde\phi_2^{y_0,DR}),M_{b_0}\widetilde\phi_0^{y_0,DR}\>\right|&\less D^{3m}(DR)^n\less (1+t)^{n+3m}R^n.
\end{align*}
This completes the proof.
\end{proof}

\begin{lemma}\label{l:acckernelconditions}
Let $b_0,b_1,b_2\in L^\infty$ be para-accretive functions.  Suppose $T$ is an bilinear C-Z operator associated to $b_0,b_1,b_2$ with standard kernel $K$, and that $M_{b_0}T(M_{b_1}\,\cdot\,,M_{b_2}\,\cdot\,)\in WBP$.  Also let $S_k^{b_i}$ be approximations to the identity with respect to $b_i$ and $D_k^{b_0}=S_{k+1}^{b_0}-S_k^{b_0}$ with compactly supported kernels $s_k^{b_i}$ and $d_k^{b_i}$ for $k\in\Z$.  Then
\begin{align*}
\theta_k(x,y_1,y_2)=\<T\(b_1s_k^{b_1}(\cdot,y_1),b_2s_k^{b_2}(\cdot,y_2)\),b_0d_k^{b_0}(x,\cdot)\>
\end{align*}
is a collection of Littlewood-Paley square function kernels of type $SBLPK$.  Furthermore $\theta_k$ satisfies
\begin{align*}
&\int_{\R^n}\theta_k(x,y_1,y_2)b_0(x)dx=0
\end{align*}
for all $y_1,y_2\in\R^n$.
\end{lemma}

\begin{proof}
Fix $x,y_1,y_2\in\R^n$ and $k\in\Z$.  We split estimate \eqref{BLPker1} into two cases:  $|x-y_1|+|x-y_2|\leq 2^{3-k}$ and $|x-y_1|+|x-y_2|>2^{3-k}$.  Note that
\begin{align*}
\phi_1(u)=s_k^{b_1}(u+2^ky_1,2^ky_1)
\end{align*}
is a normalized bump up to a constant multiple and $s_k^{b_1}(u,y_1)=2^{-kn}\phi_1^{y_1,2^{-k}}(u)$.  Likewise $s_k^{b_2}(u,y_2)=2^{-kn}\phi_2^{y_2,2^{-k}}(u)$ and $d_k^{b_0}(x,u)=2^{-kn}\phi_0^{x,2^{-k}}(u)$ where $\phi_0$ and $\phi_2$ are normalize bumps up to a constant multiple.  Then
\begin{align*}
|\theta_k(x,y_1,y_2)|&=\left|\<T\(b_1s_k^{b_1}(\cdot,y_1),b_2s_k^{b_2}(\cdot,y_2)\),b_0d_k^{b_0}(x,\cdot)\>\right|\\
&=2^{3kn}\left|\<T\(b_1\phi_1^{y_1,2^{-k}},b_2\phi_2^{y_2,2^{-k}}\),b_0\phi_0^{x,2^{-k}}\>\right|\less 2^{2kn}
\end{align*}
Now if we assume that $|x-y_1|+|x-y_2|>2^{3-k}$, then it follows that $|x-y_{i_0}|>2^{2-k}$ for at least one $i_0\in\{1,2\}$ and hence
\begin{align*}
\supp(d_k^{b_0}(x,\cdot))\cap\supp(s_k^{b_i}(\cdot,y_1))\cap\supp(s_k^{b_i}(\cdot,y_2))\subset B(x,2^{-k})\cap B(y_{i_0},2^{-k})=\emptyset.
\end{align*}
Therefore, we can estimate $\theta_k$ the kernel representation of $T$ in the following way
\begin{align*}
|\theta_k&(x,y_1,y_2)|\\
&=\left|\int_{\R^{3n}}(K(u_0,u_1,u_2)-K(x,u_1,u_2))b(u_0)d_k^{b_0}(x,u_0)\prod_{i=1}^2b_i(u_i)s_k^{b_i}(u_i,y_i)du_0\,du_1\,du_2\right|\\
&\less\int_{|x-u_0|<2^{-k}}\int_{|y_1-u_1|<2^{-k}}\int_{|y_2-u_2|<2^{-k}}\frac{|u_0-x|^\gamma \;2^{3nk}du_0\,du_1\,du_2}{(|x-u_1|+|x-u_2|)^{2n+\gamma}}\\
&\less\int_{|x-u_0|<2^{-k}}\int_{|y_1-u_1|<2^{-k}}\int_{|y_2-u_2|<2^{-k}}\frac{2^{-\gamma k}\;2^{3nk}du_0\,du_1\,du_2}{(2^{-k}+|x-y_1|+|x-y_2|)^{2n+\gamma}}\\
&\less\frac{2^{-\gamma k}}{(2^{-k}+|x-y_1|+|x-y_2|)^{2n+\gamma}}\\
&\less\Phi_k^{n+\gamma/2}(x-y_1)\Phi_k^{n+\gamma/2}(x-y_2).
\end{align*}
For \eqref{BLPker2}, note that by the continuity from $b_1C_0^\delta\times b_2C_0^\delta$ into $(b_0C_0^\delta)'$ and that $S_k^b=P_kM_{(P_kb)^{-1}}P_k$ has a $C_0^\infty$ kernel, we have for $\alpha\in\N_0^n$ with $|\alpha|=1$
\begin{align*}
|\partial_x^\alpha\theta_k(x,y,z)|&=\left|\<T\(b_1s_k^{b_1}(\cdot,y_1),b_2s_k^{b_2}(\cdot,y_2)\),b_0\partial_x^\alpha(d_k(x,\cdot))\>\right|\less2^k2^{2kn}.
\end{align*}
Estimate \eqref{BLPker2} easily follows in light of Remark \ref{r:equivkernelcond}.  By symmetry, it follows that $\{\theta_k\}$ is a collection of smooth bilinear Littlewood-Paley square function kernels.  Now we verify that $\theta_k$ has integral $0$ in the $x$ spot:  By the continuity of $T$ from $b_1C_0^\delta\times b_2C_0^\delta$ into $(b_0C_0^\delta)'$
\begin{align*}
\int_{\R^n}\theta_k(x,y_1,y_2)b_0(x)dx&=\lim_{R\rightarrow\infty}\<T(b_1s_k^{b_1}(\cdot,y_1),b_2s_k^{b_2}(\cdot,y_2)),b_0\int_{|x|<R}d_k^{b_0}(x,\cdot)b_0(x)dx\>\\
&=\lim_{R\rightarrow\infty}\<T(b_1s_k^{b_1}(\cdot,y_1),b_2s_k^{b_2}(\cdot,y_2)),\lambda_R\>
\end{align*}
where we take this to be the definition of $\lambda_R$.  Now if we take $R>2\cdot 2^{-k}$, then for $|u|<R-2^{-k}$ it follows that 
\begin{align*}
\supp(d_k^{b_0}(\cdot,u))\subset B(u,2^{-k})\subset B(0,|u|+2^{-k})\subset B(0,R),
\end{align*}
and hence for $|u|<R-2^{-k}$ we have that
\begin{align*}
\lambda_R(u)&=b_0(u)\int_{|x|<R}d_k^{b_0}(x,u)b_0(x)dx=b_0(u)D_k^{b_0\,*}b_0(u)=0.
\end{align*}
Also when $|u|>R+2^{-k}$, it follows that $\supp(d_k^{b_0}(\cdot,u))\cap B(0,R)=\emptyset$, and hence that $\lambda_R(u)=0$.  So we have $\lambda_R(x)=0$ for $|x|<R-2^{-k}$ and for $|x|> R+2^{-k}$.  Finally $||\lambda_R||_{L^\infty}\leq\sup_u||d_k^{b_0}(\cdot,u)||_{L^1}\less1$.  Since $\supp(d_k^{b_0}(x,\cdot))\subset B(0,R+2^{-k})\backslash B(0,R-2^{-k})$, it follows that for $R>4(2^{-k}+|y_1|)$, we may use the integral representation
\begin{align*}
&\hspace{-1cm}\left|\<T(b_1s_k^{b_1}(\cdot,y_1),b_2s_k^{b_2}(\cdot,y_2)),\lambda_R\>\right|\\
&\leq\int_{\R^{3n}}|K(u,v_1,v_2)b_1(v_1)s_k^{b_1}(v_1,y_1)b_2(v_2)s_k^{b_2}(v_2,y_2)\lambda_R(u)|du\,dv_1\,dv_2\\
&\less\int_{|v_2-y_2|<2^{-k}}\int_{|v_1-y_1|<2^{-k}}\int_{\supp(\lambda_R)}\frac{2^{2kn}}{(|u-v_1|+|u-v_2|)^{2n}}du\,dv_1\,dv_2\\
&\leq\int_{|v_2-y_2|<2^{-k}}\int_{|v_1-y_1|<2^{-k}}\int_{\supp(\lambda_R)}\frac{2^{2kn}}{(|u|-|v_1-y_1|-|y_1|)^{2n}}du\,dv_1\,dv_2\\
&\leq\int_{|v_2-y_2|<2^{-k}}\int_{|v_1-y_1|<2^{-k}}\int_{\supp(\lambda_R)}\frac{2^{2kn}}{(R-2^{-k}-|v_1-y_1|-|y_1|)^{2n}}du\,dv_1\,dv_2\\
&\leq\int_{|v_2-y_2|<2^{-k}}\int_{|v_1-y_1|<2^{-k}}\int_{\supp(\lambda_R)}\frac{2^{2kn}}{R^{2n}}du\,dv_1\,dv_2\\
&\less|\supp(\lambda_R)|R^{-2n}\\
&\less 2^{-k}R^{-(n+1)}.
\end{align*}
This tends to zero as $R\rightarrow\infty$.  Hence $\theta_k(x,y_1,y_2)$ has integral zero in the $x$ variable.
\end{proof}

\subsection{Reduced Bilinear T(b) Theorem}

It has become a standard argument in T1 and Tb theorems to first prove a reduced version, see e.g. \cite{DJ}, \cite{DJS}, and \cite{Hart2}.  The general idea of the argument is to first assume a stronger $Tb=0$ cancellation condition, and then prove that an operator satisfying the weaker $Tb\in BMO$ cancellation condition is a perturbation of an operator satisfying the stronger cancellation condition.  More precisely this is done through a paraproduct operator, which we will construct later in this section.  First we state and prove our reduced Tb theorem.

\begin{theorem}\label{t:reducedTb}
Let $T$ be an bilinear C-Z operator associated to para-accretive functions $b_0,b_1,b_2$.  If $M_{b_0}T(M_{b_1}\,\cdot\,,M_{b_2}\,\cdot\,)\in WBP$ and 
\begin{align*}
M_{b_0}T(b_1,b_2)=M_{b_1}T^{*1}(b_0,b_2)=M_{b_2}T^{*2}(b_1,b_0)=0,
\end{align*}
then $T$ can be extended to a bounded linear operator from $L^{p_1}\times L^{p_2}$ into $L^p$ for all $1<p_1,p_2<\infty$ satisfying \eqref{Holder}.
\end{theorem}

Note that in the hypothesis of Theorem \ref{t:reducedTb}, we take $M_{b_0}T(b_1,b_2)=0$ in the sense of Definition \ref{d:Tb}:  For appropriate $\eta_R^1$, $\eta_R^2$ and all $\phi\in C_0^\delta$ such that $b_0\phi$ has mean zero
\begin{align*}
\lim_{R\rightarrow\infty}\<T(\eta_R^1b_1,\eta_R^2b_2),b_0\phi\>=0.
\end{align*}
The meaning of $M_{b_1}T^{*1}(b_0,b_2)=M_{b_2}T^{*2}(b_1,b_0)=0$ are expressed in a similar way.

\begin{proof}
Let $T$ be as in the hypothesis, $1<p,p_1,p_2<\infty$ satisfy \eqref{Holder}, and $f_0,f_1,f_2\in C_0^1$ such that $b_if_i$ have mean zero.  Then by Proposition \ref{p:smoothaccconvergence} and the continuity of $T$ from $b_1C_0^\delta\times b_2C_0^\delta$ into $(b_0C_0^\delta)'$, it follows that
\begin{align*}
|\<T(b_1f_1,b_2f_2),b_0f_0\>|&=\lim_{N\rightarrow\infty}\left|\<T(M_{b_1}S_N^{b_1}M_{b_1}f_1,M_{b_2}S_N^{b_2}M_{b_2}f_2),M_{b_0}S_N^{b_0}M_{b_0}f_0\>\right.\\
&\hspace{1.5cm}-\left.\<T(M_{b_1}S_{-N}^{b_1}M_{b_1}f_1,M_{b_2}S_{-N}^{b_2}M_{b_2}f_2),M_{b_0}S_{-N}^{b_0}M_{b_0}f_0\>\right|\\
&=\lim_{N\rightarrow\infty}\left|\sum_{k=-N}^{N-1}\<T(M_{b_1}S_{k+1}^{b_1}M_{b_1}f_1,M_{b_2}S_{k+1}^{b_2}M_{b_2}f_2),M_{b_0}S_{k+1}^{b_0}M_{b_0}f_0\>\right.\\
&\left.\phantom{\sum_k}\hspace{1.05cm}-\<T(M_{b_1}S_k^{b_1}M_{b_1}f_1,M_{b_2}S_k^{b_2}M_{b_2}f_2),M_{b_0}S_k^{b_0}M_{b_0}f_0\>\right|\\
&\leq\sum_{k\in\Z}\left|\int_{\R^n}\Theta_k^0(b_1f_1,b_2f_2)b_0(x)f_0(x)dx\right|\\
&\hspace{1.5cm}+\left|\int_{\R^n}\Theta_k^1(b_0f_0,b_2f_2)b_1(x)f_1(x)dx\right|\\
&\hspace{1.5cm}+\left|\int_{\R^n}\Theta_k^2(b_1f_1,b_0f_0)b_2(x)f_2(x)dx\right|.
\end{align*}
where
\begin{align*}
\Theta_k^0(f_1,f_2)&=D_k^{b_0}M_{b_0}T(M_{b_1}S_{k+1}^{b_1}f_1,M_{b_2}S_{k+1}^{b_2}f_2),\\
\Theta_k^1(f_1,f_2)&=D_k^{b_1}M_{b_1}T^{*1}(M_{b_0}S_k^{b_0}f_1,M_{b_2}S_k^{b_2}f_2),\\
\Theta_k^2(f_1,f_2)&=D_k^{b_2}M_{b_2}T^{*2}(M_{b_1}S_{k+1}^{b_1}f_1,M_{b_0}S_k^{b_0}f_2).
\end{align*}
We focus on $\Theta_k^0=\Theta_k$ to simplify notation; the other terms are handled in the same way.  Since $M_{b_0}T(M_{b_1}\;\cdot,M_{b_2}\;\cdot)\in WBP$ and $T$ has a standard kernel, it follows from Lemma \ref{l:acckernelconditions} that $\{\theta_k\}\in SBLPK$ and $\theta_k(x,y_1,y_2)b_0(x)$ has mean zero in the $x$ variable for all $y_1,y_2\in\R^n$.  Now we show that $\Theta_k(b_1,b_2)=0$, which follows from the assumption that $M_{b_0}T(b_1,b_2)=0$:
\begin{align*}
\Theta_k(b_1,b_2)(x)&=\int_{\R^{2n}}\<M_{b_0}T\(M_{b_1}s_k^{b_1}(\cdot,y_1)b_1(y_1),M_{b_2}s_k^{b_2}(\cdot,y_2)b_2(y_2)\),d_k^{b_0}(x,\cdot)\>dy\\
&=\lim_{R\rightarrow\infty}\<T\(b_1\eta_R^1,b_2\eta_R^2\),b_0d_k^{b_0}(x,\cdot)\>=0,
\end{align*}
where
\begin{align*}
\eta_R^i(u)=\int_{|y|<R}s_k^{b_i}(u,y)b_i(y)dy.
\end{align*}
We've used that $M_{b_0}T(b_1,b_2)=0$, and that $\eta_R^i\in C^\infty$, $\eta_R^i\equiv1$ on $B(0,R)$, and $\supp(\eta_R^i)\subset B(0,2R)$ for $R$ sufficiently large.  Then by Theorem \ref{t:mdualbound}, it follows that 
\begin{align*}
\sum_{k\in\Z}\left|\<\Theta_k^0(M_{b_1}f_1,M_{b_2}f_2),M_{b_0}f_0\>\right|\less||f_0||_{L^{p'}}||f_1||_{L^{p_1}}||f_2||_{L^{p_2}}.
\end{align*}
A similar argument holds for $\Theta_k^i$ with $i=1,2$ again taking advantage of the facts $\frac{1}{p'}+\frac{1}{p_2}=\frac{1}{p_1'}$ and $\frac{1}{p_1}+\frac{1}{p'}=\frac{1}{p_2'}$.  Therefore $T$ can be extended to a bounded operator from $L^{p_1}\times L^{p_2}$ into $L^p$ for all $1<p,p_1,p_2<\infty$ satisfying \eqref{Holder}.
\end{proof}

\subsection{A Para-Product Construction}

In the original proof of the T1 theorem, David-Journ\'e \cite{DJ} used the Bony paraproduct \cite{B} to pass from their reduced T1 theorem to the full T1 theorem.  Following the same idea, David-Jounrn\'e-Semmes \cite{DJS} proved the Tb theorem by constructing a para-accretive version of the Bony paraproduct.  In \cite{Hart2}, we constructed a bilinear Bony-type paraproduct, which allowed us to transition from a reduce bilinear T1 theorem to a full T1 theorem.  Here we construct a bilinear paraproduct in a para-accretive function setting.  First we prove a quick lemma, which essentially appears in a work by Benyi-Maldonado-Nahmod-Torres \cite{BMNT} and is a bilinear version of an observation made by Coifman-Meyer \cite{CM1}.

\begin{lemma}\label{l:CZkernel}
Suppose $\{\theta_k\}\in SBLPK$ with decay parameter $N>2n$, and define $K:\R^{3n}\backslash\{(x,x,x):x\in\R^n\}\rightarrow\C$ 
\begin{align*}
K(x,y_1,y_2)=\sum_{k\in\Z}\theta_k(x,y_1,y_2).
\end{align*}
Then $K$ is a bilinear standard Calder\'on-Zygmund kernel.
\end{lemma}

\begin{proof}
To prove the size estimate, we take $d=|x-y_1|+|x-y_2|\neq0$ and compute
\begin{align*}
|K(x,y_1,y_2)|&\less\sum_{k\in\Z}\frac{2^{2kn}}{(1+2^k|x-y_1|)^{N+\gamma}(1+2^k|x-y_2|)^{N+\gamma}}\\
&\less\sum_{2^k\leq d^{-1}}2^{2kn}+\sum_{2^k>d^{-1}}\frac{2^{2kn}}{(2^kd)^{N+\gamma}}\less d^{-2n}.
\end{align*}
For the regularity in $x$, we take $x,x',y_1,y_2\in\R^n$ with $|x-x'|<\max(|x-y_1|,|x-y_2|)/2$ and define $d'=|x'-y_1|+|x'-y_2|$.  Then
\begin{align*}
|K(x,y_1,y_2)-K(x',y_1,y_2)|&\less\sum_{k\in\Z}\frac{(2^k|x-x'|)^{\gamma}2^{2kn}}{(1+2^k|x-y_1|)^{N+\gamma}(1+2^k|x-y_2|)^{N+\gamma}}\\
&\hspace{1cm}+\sum_{k\in\Z}\frac{(2^k|x-x'|)^{\gamma}2^{2kn}}{(1+2^k|x'-y_1|)^{N+\gamma}(1+2^k|x'-y_2|)^{N+\gamma}}\\
&=I+II.
\end{align*}
We first bound $I$ by $|x-x'|^\gamma$ times
\begin{align*}
\sum_{2^k\leq d^{-1}}2^{k(2n+\gamma)}+\sum_{2^k>d^{-1}}\frac{2^{k(2n+\gamma)}}{(2^kd)^{N+\gamma}}\less d^{-(2n+\gamma)}+d^{-(N+\gamma)}\sum_{2^k>d^{-1}}2^{k(2n-N)}\less d^{-(2n+\gamma)}.
\end{align*}
By symmetry, it follows that $II\less|x-x'|^\gamma d'^{-(2n+\gamma)}$, but since $|x-x'|<\max(|x-y_1|,|x-y_2|)/2$, without loss of generality say $|x-y_1|\geq|x-y_2|$ it follows that
\begin{align*}
II\less\frac{|x-x'|^\gamma}{(|x'-y_1|+|x'-y_2|)^{2n+\gamma}}\leq\frac{|x-x'|^\gamma}{(|x-y_1|-|x-x'|)^{2n+\gamma}}\less\frac{|x-x'|^\gamma}{|x-y_1|^{2n+\gamma}}\less\frac{|x-x'|^\gamma}{d^{2n+\gamma}}
\end{align*}
With a similar computation for $y_1,y_2$, it follows that $K$ is a standard kernel.
\end{proof}

\begin{theorem}\label{t:paraproduct}
Given para-accretive functions $b_0,b_1,b_2\in L^\infty$ and $\beta\in BMO$, there exists a bilinear Calder\'on-Zygmund operator $L$ that is bounded from $L^{p_1}\times L^{p_2}$ into $L^p$ for all $1<p_1,p_2<\infty$ satisfying \eqref{Holder} with $p=2$ such that $M_{b_0}L(b_1,b_2)=\beta$, $M_{b_1}L^{*1}(b_0,b_2)=M_{b_2}L^{*2}(b_1,b_0)=0$.
\end{theorem}

\begin{proof}
Let $b_0,b_1,b_2$ be para-accretive functions, and $S_k^{b_i}$, $D_k^{b_i}$, and $\widetilde D_k^{b_i}$ be the approximation to identity and reproducing formula operators with respect to $b_i$ for $i=0,1,2$ that have compactly supported kernels as defined in Remark \ref{r:accapptoid} and Theorem \ref{t:Han}.  Define
\begin{align*}
&L(f_1,f_2)=\sum_{k\in\Z}L_k(f_1,f_2)=\sum_{k\in\Z}D_k^{b_0}M_{b_0}\((\widetilde D_k^{b_0\,*}M_{b_0}\beta)(S_k^{b_1}f_1)(S_k^{b_2}f_2)\)\\
&\ell(x,y)=\sum_{k\in\Z}\ell_k(x,y)=\sum_{k\in\Z}\int_{\R^n}d_k^{b_0}(x,u)b_0(u)\widetilde D_k^{b_0\,^*}M_{b_0}\beta(u)s_k^{b_1}(u,y_1)s_k^{b_2}(u,y_2)du.
\end{align*}
It follows that $L$ is bounded from $L^{p_1}\times L^{p_2}$ into $L^2$ for all $1<p_1,p_2<\infty$ satisfying $\frac{1}{2}=\frac{1}{p_1}+\frac{1}{p_2}$:
\begin{align*}
\left|\int_{\R^n}L(f_1,f_2)(x)f_0(x)dx\right|&\leq\sum_{k\in\Z}\left|\int_{\R^n}\widetilde D_k^{b_0\,*}M_{b_0}\beta(x)S_k^{b_1}f_1(x)S_k^{b_2}f_2(x)D_k^{b_0}f_0(x)b_0(x)dx\right|\\
&\hspace{-1cm}\less\norm{\(\sum_{k\in\Z}|M_{\widetilde D_k^{b_0\,*}M_{b_0}\beta}S_k^{b_1}f_1S_k^{b_2}f_2|^2\)^\frac{1}{2}}{L^2}\norm{\(\sum_{k\in\Z}|D_k^{b_0}f_0|^2\)^\frac{1}{2}}{L^2}\\
&\hspace{-1cm}\less\(\int_{\R^n}\sum_{k\in\Z}\[\Phi_k^N*|f_1|(x)\Phi_k^N*|f_2|(x)\]^2|\widetilde D_k^{b_0\,*}M_{b_0}\beta(x)|^2dx\)^\frac{1}{2}\norm{f_0}{L^2}\\
&\hspace{-1cm}\less\(\int_{\R^n}\sum_{k\in\Z}\[\Phi_k^N*|f_1|(x)\]^{p_1}|\widetilde D_k^{b_0\,*}M_{b_0}\beta(x)|^2dx\)^\frac{1}{p_1}\\
&\hspace{-1cm}\hspace{1cm}\times\(\int_{\R^n}\sum_{k\in\Z}\[\Phi_k^N*|f_2|(x)\]^{p_2}|\widetilde D_k^{b_0\,*}M_{b_0}\beta(x)|^2dx\)^\frac{1}{p_2}\norm{f_0}{L^2}\\
&\hspace{-1cm}\less\norm{f_0}{L^2}\norm{f_1}{L^{p_1}}\norm{f_2}{L^{p_2}}.
\end{align*}
Note that in the last line we apply the discrete version of a well-known result relating Carleson measure and square functions due to Carleson \cite{C} and Jones \cite{J} for the Carleson measure
\begin{align*}
d\mu(x,t)=\sum_{k\in\Z}|\widetilde D_k^{b_0\,*}b_0\beta(x)|^2\delta_{t=2^{-k}}.
\end{align*}
For details of the discrete version of this result, see for example the book by Grafakos \cite{G}, Theorems 7.3.7 and 7.3.8(c).  This proves that $L$ is bounded from $L^{p_1}\times L^{p_2}$ into $L^2$ for all $1<p_1,p_2<\infty$ satisfying \eqref{Holder} with $p=2$.  It is easy to check that $\{\ell_k\}\in SBLPK$ with size index $N>2n$:  since $d_k^{b_0}$ and $s_k^{b_i}$ are compactly supported kernels, for $i=1,2$ it follows that 
\begin{align*}
|\ell_k(x,y_1,y_2)|&\leq||b_0\widetilde D_k^{b_0\,*}M_{b_0}\beta||_{L^\infty}\int_{\R^n}|d_k^{b_0}(x-u)s_k^{b_1}(u-y_1)s_k^{b_2}(u-y_2)|du\\
&\less2^{kn}\int_{\R^n}\Phi_k^{4(n+1)}(x-u)\Phi_k^{4(n+1)}(u-y_i)du\\
&\less2^{kn}\Phi_k^{4(n+1)}(x-y_i).
\end{align*}
Hence the size condition \eqref{BLPker1} with size index $N=2n+1$ and $\gamma=1$ follows
\begin{align*}
|\ell_k(x,y_1,y_2)|&\less\prod_{i=1}^2\(2^{kn}\Phi_k^{4(n+1)}(x-y_i)\)^\frac{1}{2}=\Phi_k^{2n+2}(x-y_1)\Phi_k^{2n+2}(x-y_2).
\end{align*}
The regularity estimates \eqref{BLPker2}-\eqref{BLPker4} follow easily from the regularity of $d_k^{b_0}$, $s_k^{b_1}$, and $s_k^{b_2}$.  Then by Lemma \ref{l:CZkernel}, $L$ has a standard Calder\'on-Zygmund kernel $\ell$.  It follows from a result of Grafakos-Torres \cite{GT1} that $L$ is bounded from $L^{p_1}\times L^{p_2}$ into $L^p$ where $1<p_1,p_2<\infty$ satisfy \eqref{Holder}.  Next we compute $M_{b_0}L(b_1,b_2)$:  Let $\delta>0$, $\phi\in C_0^\delta$ such that $\supp(\phi)\subset B(0,N)$ and $b_0\phi$ has mean zero.  Let $\eta_R(x)=\eta(x/R)$ where $\eta\in C_0^\infty$ satisfies $\eta\equiv1$ on $B(0,1)$, and $\supp(\eta)\subset B(0,2)$.  Then
\begin{align*}
&\<L(b_1,b_2),b_0\phi\>\\
&\hspace{.5cm}=\lim_{R\rightarrow\infty}\sum_{2^{-k}<R/4}\int_{\R^n} \widetilde D_k^{b_0\,^*}M_{b_0}\beta(x)S_k^{b_1}M_{b_1}\eta_R(x)S_k^{b_2}M_{b_2}\eta_R(x)M_{b_0}D_k^{b_0}(b_0\phi)(x)dx\\
&\hspace{1.5cm}+\lim_{R\rightarrow\infty}\sum_{2^{-k}\geq R/4}\int_{\R^n}\widetilde D_k^{b_0\,^*}M_{b_0}\beta(x)S_k^{b_1}M_{b_1}\eta_R(x)S_k^{b_2}M_{b_2}\eta_R(x)M_{b_0}D_k^{b_0}(b_0\phi)(x)dx,
\end{align*}
where we may write this only if the two limits on the right hand side of the equation exist.  As we are taking $R\rightarrow\infty$ and $N$ is a fixed quantity determined by $\phi$, without loss of generality assume that $R>2N$.  Note that for $2^{-k}\leq R/4$ and $|x|<N+2^{-k}$,
$$\supp(s_k^{b_i}(x,\cdot))\subset B(x,2^{-k})\subset B(0,N+2^{1-k})\subset B(0,R).$$
Since $\eta_R\equiv1$ on $B(0,R)$, it follows that $S_k^{b_i}M_{b_i}\eta_R(x)=1$ for all $|x|<N+2^{-k}$ when $2^{-k}\leq R/4$.  Therefore
\begin{align*}
\lim_{R\rightarrow\infty}\sum_{2^{-k}<R/4}\int_{\R^n} \widetilde D_k^{b_0\,^*}M_{b_0}\beta(x)M_{b_0}D_k^{b_0}(b_0\phi)(x)dx&=\int_{\R^n}\sum_{k\in\Z}M_{b_0}\widetilde D_k^{b_0}M_{b_0}D_kM_{b_0}\phi(x)\beta(x)dx\\
&=\<\beta,b_0\phi\>.
\end{align*}
Here we use the convergence of the accretive type reproducing formula in $H^1$ from Corollary \ref{c:accH1}.  For any $k\in\Z$, we have the estimates
\begin{align}
&||S_k^{b_i}M_{b_i}\eta_R||_{L^1}\less||\eta_R||_{L^1}\less R^n,\label{estimate1}\\
&||S_k^{b_i}M_{b_i}\eta_R||_{L^\infty}\less||\eta_R||_{L^\infty}=1,\label{estimate2}
\end{align}
and for any $x\in\R^n$
\begin{align*}
|D_k^{b_0}M_{b_0}\phi(x)|&\leq\int_{\R^n}|d_k^{b_0}(x,y)-d_k^{b_0}(x,0)|\,|b_0(y)\phi(y)|dy\\
&\less\int_{\R^n}(2^k|y|)^\gamma|\phi(y)|dy\less N^\gamma||\phi||_{L^1}2^{k(n+\gamma)}.
\end{align*}
Here we know that $\{d_k^{b_0}\}\in LPK$, so without loss of generality we take the corresponding smoothness parameter $\gamma\leq\delta$.  Later we will use that $\gamma\leq\delta$ implies $\phi\in C_0^\delta\subset C_0^\gamma$, so we have that $|\phi(x)-\phi(y)|\less|x-y|^\gamma$.  Therefore
\begin{align}
&\sum_{2^{-k}>R/4}\int_{\R^n}|\widetilde D_k^{b_0\,^*}M_{b_0}\beta(x)S_k^{b_1}M_{b_1}\eta_R(x)S_k^{b_2}M_{b_2}\eta_R(x)M_{b_0}D_k^{b_0}(b_0\phi)(x)|dx\notag\\
&\hspace{1cm}\leq\sum_{2^{-k}>R/4}||\widetilde D_k^{b_0\,^*}M_{b_0}\beta||_{L^\infty}||S_k^{b_1}M_{b_1}\eta_R||_{L^1}||S_k^{b_2}M_{b_2}\eta_R||_{L^\infty}||M_{b_0}D_k^{b_0}(b_0\phi)||_{L^\infty}\notag\\
&\hspace{1cm}\less R^n\sum_{2^{-k}>R/4}2^{k(n+\gamma)}\less R^{-\gamma}.\label{secondterm}
\end{align}
Hence the second limit above exists and tends to $0$ as $R\rightarrow\infty$.  Then $\<L(b_1,b_2),b_0\phi\>=\<\beta,b_0\phi\>$ for all $\phi\in  C_0^\delta$ such that $b_0\phi$ has mean zero and hence $M_{b_0}L(b_1,b_0)=\beta$ as defined in Section 2.  Again for any $\phi\in C_0^\delta$ such that $b_1\phi$ has mean zero and $\supp(\phi)\subset B(0,N)$, we have
\begin{align*}
&\left|\<L^{1*}(b_0,b_2),b_1\phi\>\right|\\
&\hspace{1cm}=\lim_{R\rightarrow\infty}\left|\sum_{k\in\Z}\int_{\R^n}\widetilde D_k^{b_0\,*}M_{b_0}\beta(x)S_k^{b_1}M_{b_1}\phi(x)S_k^{b_2}M_{b_2}\eta_R(x)D_k^{b_0}M_{b_0}\eta_R(x)b_0(x)dx\right|\\
&\hspace{1cm}\less\lim_{R\rightarrow\infty}\sum_{k\in\Z}||\widetilde D_k^{b_0\,*}M_{b_0}\beta||_{L^\infty}||S_k^{b_1}M_{b_1}\phi||_{L^1}||S_k^{b_2}M_{b_2}\eta_R||_{L^\infty}||D_k^{b_0}M_{b_0}\eta_R||_{L^\infty}\\
&\hspace{1cm}\less\lim_{R\rightarrow\infty}\sum_{k\in\Z}||S_k^{b_1}M_{b_1}\phi||_{L^1}||D_k^{b_0}M_{b_0}\eta_R||_{L^\infty}.
\end{align*}
We will now show that $||S_k^{b_1}M_{b_1}\phi||_{L^1}||D_k^{b_0}M_{b_0}\eta_R||_{L^\infty}$ bounded by a in integrable function in $k$ (i.e. summable) independent of $R$, so that we can bring the limit in $R$ inside the sum.  To do this we start by estimating
\begin{align*}
|S_k^{b_1}M_{b_1}\phi(x)|&\leq\int_{\R^n}|s_k^{b_1}(x,y)-s_k^{b_1}(x,0)|\,|\phi(y)b_1(y)|dy\\
&\leq N^\gamma||\phi||_{L^1}||b_1||_{L^\infty}2^{\gamma k}\(\Phi_0^N(x)+\Phi_k^N(x)\)
\end{align*}
and so $||S_k^{b_1}M_{b_1}\phi||_{L^1}\less2^{\gamma k}$.  We also have that $||S_k^{b_1}M_{b_1}\phi||_{L^1}\less||\phi||_{L^1}\less1$, so \\$||S_k^{b_1}M_{b_1}\phi||_{L^1}\less\min(1,2^{\gamma k})$.  Also
\begin{align*}
|D_k^{b_0}M_{b_0}\eta_R(x)|&\leq\int_{\R^n}|d_k^{b_0}(x,y)|\,|\eta_R(y)-\eta_R(x)|\,|b_0(y)|dy\\
&\less2^{-\gamma k}R^{-\gamma}\int_{\R^n}\Phi_k^{N+\gamma}(x-y)(2^k|x-y|)^\gamma dy\less2^{-\gamma k}R^{-\gamma}.
\end{align*}
It follows that $||D_k^{b_0}M_{b_0}\eta_R||_{L^\infty}\less||\eta_R||_{L^\infty}\less1$, and hence $||D_k^{b_0}M_{b_0}\eta_R||_{L^\infty}\less\min(1,2^{-\gamma k})$.  So when $R>1$, we have 
\begin{align*}
||D_k^{b_0}M_{b_0}\eta_R||_{L^\infty}||S_k^{b_1}M_{b_1}\phi||_{L^1}\less\min(2^{-\gamma k }R^{-\gamma},2^{\gamma k})\leq2^{-\gamma|k|},
\end{align*}
and hence by dominated convergence
\begin{align*}
\left|\<L^{1*}(b_0,b_2),b_1\phi\>\right|&\less\sum_{k\in\Z}\lim_{R\rightarrow\infty}||S_k^{b_1}M_{b_1}\phi||_{L^1}||D_k^{b_0}M_{b_0}\eta_R||_{L^\infty}\less\sum_{k\in\Z}\lim_{R\rightarrow\infty}2^{-k\gamma}R^{-\gamma}=0
\end{align*}
Then $M_{b_1}L^{*1}(b_0,b_2)=0$ and a similar argument shows that $M_{b_2}L^{*2}(b_1,b_0)=0$, which concludes the proof.
\end{proof}

Now to complete the proof of Theorem \ref{t:fullTb} is a standard argument using the reduced Tb Theorem \ref{t:reducedTb} and paraproducts constructioned in Theorem \ref{t:paraproduct}.

\begin{proof}
Assume that $M_{b_0}T(M_{b_1}\,\cdot,M_{b_2}\,\cdot)$ satisfies the weak boundedness property and 
\begin{align*}
M_{b_0}T(b_1,b_2),M_{b_1}T^{*1}(b_0,b_2),M_{b_2}T^{*2}(b_1,b_0)\in BMO.
\end{align*}
By Theorem \ref{t:paraproduct}, there exist bounded bilinear Calder\'on-Zygmund operators $L_i$ such that 
\begin{align*}
&M_{b_0}L_0(b_1,b_2)=M_{b_0}T(b_1,b_2),& &M_{b_1}L_0^{*1}(b_0,b_2)=M_{b_2}L_0^{*2}(b_1,b_0)=0,&\\
&M_{b_1}L_1^{*1}(b_0,b_2)=M_{b_1}T^{*1}(b_0,b_2),& &M_{b_0}L_1(b_1,b_2)=M_{b_2}L_1^{*2}(b_1,b_0)=0,&\\
&M_{b_2}L_2^{*2}(b_1,b_0)=M_{b_2}T^{*2}(b_1,b_0),& &M_{b_1}L_2^{*1}(b_0,b_2)=M_{b_0}L_2(b_1,b_2)=0.&
\end{align*}
Now define the operator
\begin{align*}
S=T-\sum_{i=0}^2L_i,
\end{align*}
which is continuous from $b_1C_0^\delta\times b_2C_0^\delta$ into $(b_0C_0^\delta)'$.  Also $M_{b_0}S(M_{b_1}\,\cdot,M_{b_2}\,\cdot)$ satisfies the weak boundedness property since $M_{b_0}T(M_{b_1}\,\cdot,M_{b_2}\,\cdot)$ and $M_{b_0}L_i(M_{b_1}\,\cdot,M_{b_2}\,\cdot)$ for $i=0,1,2$ do.  Finally we have
\begin{align*}
&M_{b_0}S(b_1,b_2)=M_{b_0}T(b_1,b_2)-\sum_{i=0}^2M_{b_0}L_i(b_1,b_2)=0&\\
&M_{b_1}S^{*1}(b_0,b_2)=M_{b_1}T^{*1}(b_0,b_2)-\sum_{i=0}^2M_{b_1}L_i^{*1}(b_0,b_2)=0&\\
&M_{b_2}S^{*2}(b_1,b_0)=M_{b_2}T^{*2}(b_1,b_0)-\sum_{i=0}^2M_{b_2}L_i^{*2}(b_1,b_0)=0&
\end{align*}
Then by Theorem \ref{t:reducedTb}, $S$ can be extended to a bounded linear operator from $L^{p_1}\times L^{p_2}$ into $L^p$ for all $1<p,p_1,p_2<\infty$ satisfying \eqref{Holder}.  Therefore $T$ is bounded on the same indices, and by results from \cite{GT1}, $T$ is also bounded without restriction on $p$.  The converse is also a well-known result from \cite{GT1}.
\end{proof}

\section{Application to a Bilinear Reisz Transforms defined on Lipschitz Curves}

In this section, we apply the bilinear Tb theorem proved above to a bilinear version of the Reisz transforms along Lipschitz curves in the complex plane.  We prove bounds of the form 
\begin{align*}
||T(f_1,f_2)||_{L^p(\Gamma_0)}\less||f_1||_{L^{p_1}(\Gamma)}||f_2||_{L^{p_2}(\Gamma)}
\end{align*}
for parameterized Lipschitz curves $\Gamma$ and $p,p_1,p_2$ satisfying H\"older.  

We fix some notation for this section.  Let $L$ be a Lipschitz function with small Lipschitz constant $\lambda<1$, and define the parameterization $\gamma(x)=x+iL(x)$ of the curve $\Gamma=\{\gamma(x):x\in\R\}$.  Define $L^p(\Gamma)$ to be the collection of all measurable functions $f:\Gamma\rightarrow\C$ such that
\begin{align*}
||f||_{L^p(\Gamma)}=\(\int_\Gamma|f(z)|^p|dz|\)^\frac{1}{p}=\(\int_\R|f(\gamma(x))|^p|\gamma'(x)|dx\)^\frac{1}{p}<\infty.
\end{align*}

The applications in this section are in part motivated by the proof of $L^p$ bounds for the Cauchy integral using the Tb theorem of David-Journ\'e-Semmes \cite{DJS}.  We define the Cauchy integral operator for appropriate $g:\Gamma\rightarrow\C$ and $z\in\Gamma$
\begin{align*}
\mathcal C_\Gamma g(z)=\lim_{\epsilon\rightarrow0^+}\int_\Gamma\frac{g(\xi)d\xi}{(\xi+i\epsilon)-z},
\end{align*}
and parameterized Cauchy integral operator for $f:\R\rightarrow\C$ and $x\in\R$
\begin{align*}
\widetilde{\mathcal C}_\Gamma f(z)=\lim_{\epsilon\rightarrow0^+}\int_\R\frac{f(y)dy}{(\gamma(y)+i\epsilon)-\gamma(x)}.
\end{align*}
The bounds of $\mathcal C_\Gamma$ on $L^p(\Gamma)$ can be reduced to the bounds of $\widetilde{\mathcal C}_\Gamma$ on $L^p(\R)$.  We formally check the Tb conditions for $\widetilde {\mathcal C}_\Gamma$ with $b_0=b_1=\gamma'$ needed to apply the Tb theorem of David-Journ\'e-Semmes:  We check (1) $\widetilde {\mathcal C}_\Gamma(\gamma')\in BMO$
\begin{align*}
\widetilde {\mathcal C}_\Gamma\gamma'(x)&=\lim_{\epsilon\rightarrow0^+}\int_\R\frac{\gamma'(y)dy}{(\gamma(y)+i\epsilon)-\gamma(x)}=\lim_{\epsilon\rightarrow0^+}\int_\Gamma\frac{d\xi}{(\xi+i\epsilon)-\gamma(x)}=2\pi i
\end{align*}
and (2) $\widetilde {\mathcal C}_\Gamma^*(\gamma')\in BMO$, for appropriate $\phi\in C_0^\infty$
\begin{align*}
\widetilde {\mathcal C}_\Gamma^*\gamma'(x)&=\lim_{\epsilon\rightarrow0^+}\int_\R\frac{\gamma'(y)dy}{(\gamma(x)+i\epsilon)-\gamma(y)}=\lim_{\epsilon\rightarrow0^+}-\int_\Gamma\frac{d\xi}{(\xi-i\epsilon)-\gamma(x)}=0.
\end{align*}
The crucial role that Cauchy's formula plays in this argument is to be able to evaluate the limit from the definition of $\widetilde{\mathcal C}_\Gamma$ for nice enough input functions $\gamma'(x)f(x)$.  In our application, we use a similar argument except the role of Cauchy's integral formula is replaced with an integration by parts identity to verify the WBP and Tb conditions.

To further motivate looking at bilinear Reisz transforms along Lipschitz curve, we look at the ``flat'' bilinear Riesz tranforms, which we generate from a potential function perspective.  Consider the potential function
\begin{align*}
F(x,y_1,y_2)=\frac{1}{((x-y_1)^2+(x-y_2)^2)^{1/2}},
\end{align*}
and the kernels that it generates:
\begin{align*}
K_0(x,y_1,y_2)&=\partial_xF(x,y_1,y_2)=\frac{2x-y_1-y_2}{((x-y_1)^2+(x-y_2)^2)^{3/2}},\\
K_1(x,y_1,y_2)&=\partial_{y_1}F(x,y_1,y_2)=\frac{x-y_1}{((x-y_1)^2+(x-y_2)^2)^{3/2}},\\
K_2(x,y_1,y_2)&=\partial_{y_2}F(x,y_1,y_2)=\frac{x-y_2}{((x-y_1)^2+(x-y_2)^2)^{3/2}}.
\end{align*}
We define the bilinear Reisz transforms as principle value integrals for $f_1,f_2\in C_0^\infty$,
\begin{align*}
R_j(f_1,f_2)(x)=p.v.\int_{\R^2}K_j(x,y_1,y_2)f_1(y_1)f_2(y_2)dy_1\,dy_2.
\end{align*}
Here is it only interesting to study two of these three operators since $R_0=R_1+R_2$.  The bilinear T1 theorem of Christ-Journ\'e \cite{CJ} or Grafakos-Torres \cite{GT1} can be applied to the bilinear Riesz transforms $R_1$:  We formally check (1) $R_1(1,1)\in BMO$
\begin{align*}
R_1(1,1)(x)&=-\int_{\R^2} F(x,y_1,y_2)\partial_{y_1}(1)dy_1\,dy_2=0,
\end{align*}
(2) $R_1^{*1}(1,1)\in BMO$
\begin{align*}
R_1^{*1}(1,1)(y_1)&=\int_{\R^2}\partial_{y_1}F(x,y_1,y_2)dx\,dy_2\\
&=\int_{\R^2}(\partial_xF(x,y_1,y_2)-\partial_{y_2}F(x,y_1,y_2))dx\,dy_2\\
&=-\int_{\R^2}F(x,y_1,y_2)(\partial_x(1)-\partial_{y_2}(1))dx\,dy_2=0,
\end{align*}
and (3) $R_1^{*2}(1,1)\in BMO$
\begin{align*}
R_1^{*2}(1,1)(y_2)&=-\int_{\R^2}F(x,y_1,y_2)\partial_{y_1}(1)dx\,dy_1=0.
\end{align*}
Here we observe that the conditions $R_1(1,1),R_1^{*2}(1,1)=0$ are identical arguments and rely on the cancellation of the kernel $K_1$.  The $R_1^{*1}(1,1)$ condition relies on more than just the cancellation $K_1$; it also exploits the symmetry of the kernel via the identity $\partial_{y_1}F(x,y_1,y_2)=\partial_xF(x,y_1,y_2)-\partial_{y_2}F(x,y_1,y_2)$.  This is the general argument that we will use to prove $L^p$ bounds for bilinear Riesz transforms defined along Lipschitz curves, which we define now.

For $z,\xi_1,\xi_2\in\Gamma$, define the potential function
\begin{align*}
F_\Gamma(z,\xi_1,\xi_2)=\frac{1}{((z-\xi_1)^2+(z-\xi_2)^2)^{1/2}},
\end{align*}
and the Riesz kernels generated by $F$:
\begin{align*}
K_{\Gamma,0}(z,\xi_1,\xi_2)&=\partial_zF_\Gamma(\xi,\xi_1,\xi_2)=\frac{2z-\xi_1-\xi_2}{((z-\xi_1)^2+(z-\xi_2)^2)^{3/2}},\\
K_{\Gamma,1}(z,\xi_1,\xi_2)&=\partial_{\xi_1}F_\Gamma(z,\xi_1,\xi_2)=\frac{z-\xi_1}{((z-\xi_1)^2+(z-\xi_2)^2)^{3/2}},\\
K_{\Gamma,2}(z,\xi_1,\xi_2)&=\partial_{\xi_2}F_\Gamma(z,\xi_1,\xi_2)=\frac{z-\xi_2}{((z-\xi_1)^2+(z-\xi_2)^2)^{3/2}}.
\end{align*}
In the remainder of this section, we will keep the notation $z=\gamma(x)$, $\xi_1=\gamma(y_1)$, $\xi_2=\gamma(y_2)$, $y_0=x$, and $\xi_0=z$.  Here we define $\sqrt{\;\cdot\;}$ on $\C$ with the negative real axis for a branch cut, i.e.  for $\omega=r\,e^{i\theta}\in\C$ with $r>0$ and $\theta\in(-\pi,\pi]$, we define $\sqrt\omega=\sqrt r\,e^{i\theta/2}$.  We make this definition to be precise, but it will not cause any issues with computations since we will only evaluate $\sqrt\omega$ for $\omega\in\C$ with positive real part.

For appropriate $g_1,g_2:\Gamma\rightarrow\C$ and $z\in\Gamma$, we define
\begin{align}
C_{\Gamma,j}(g_1,g_2)(z)&=p.v.\int_{\Gamma^2}K_{\Gamma,j}(z,\xi_1,\xi_2)g_1(\xi_1)g_2(\xi_2)d\xi_1\,d\xi_2\notag\\
&=\lim_{\epsilon\rightarrow0^+}\int_{|Re(z-\xi_1)|,|Re(z-\xi_2)|>\epsilon}K_{\Gamma,j}(z,\xi_1,\xi_2)g_1(\xi_1)g_2(\xi_2)d\xi_1\,d\xi_2.\label{defnCgamma}
\end{align}
Initially we take this definition for $g_j=f_j\circ\gamma^{-1}$ for $f_j\in C_0^\infty(\R)$, $j=1,2$, but even for such $g_j$ it is not yet apparent that this limit exists.  We will establish that this limit exists, and furthermore that $C_{\Gamma,j}$ can be continuously extended to a bilinear operator from $L^{p_1}(\Gamma)\times L^{p_2}(\Gamma)$ into $L^p(\Gamma)$.  To prove these things, we will pass through ``parameterized'' versions of $F$, $K_j$, and $C_{\Gamma,j}$ for $j=0,1,2$ in the same way that David-Journ\'e-Semmes did to apply their Tb theorem to the Cauchy integral operator in \cite{DJS}:  For $x,y_1,y_2\in\R$, define
\begin{align*}
&\widetilde F_\Gamma(x,y_1,y_2)=F_\Gamma(\gamma(x),\gamma(y_1),\gamma(y_2)),& &\widetilde K_{\Gamma,j}(x,y_1,y_2)=K_\Gamma(\gamma(x),\gamma(y_1),\gamma(y_2)),&
\end{align*}
and for $f_1,f_2\in C_0^\infty(\R)$, define for $x\in\R$
\begin{align}
M_{\gamma'}\widetilde C_{\Gamma,j}(\gamma' f_1,\gamma' f_2)(x)&=p.v.\int_{\R^2}\widetilde K_j(x,y_1,y_2)f_1(y_1)f_2(y_2)\gamma'(x)\gamma'(y_1)\gamma'(y_2) dy_1\, dy_2\notag\\
&\hspace{-1cm}=\lim_{\epsilon\rightarrow0^+}\int_{|x-y_1|,|x-y_2|>\epsilon}\widetilde K_j(x,y_1,y_2)f_1(y_1)f_2(y_2)\gamma'(x)\gamma'(y_1)\gamma'(y_2) dy_1\, dy_2.\label{Riesz}
\end{align}
for $j=0,1,2$.  We begin by proving that $\widetilde C_{\Gamma,j}$ for $j=0,1,2$ is well defined, and find an absolutely convergent integral representation for it that depends on derivatives of the input functions $f_1,f_2\in\C_0^\infty(\R)$. 

\begin{proposition}\label{p:welldefined1}
Let $L$ be a Lipschitz function with Lipschitz constant $\lambda<1$ such that for almost every $x\in\R$ the limits
\begin{align*}
&\lim_{\epsilon\rightarrow0^+}\gamma'(x+\epsilon)=\gamma'(x+)& &\text{and}& &\lim_{\epsilon\rightarrow0^+}\gamma'(x-\epsilon)=\gamma'(x-)&
\end{align*}
exist.  Then $M_{\gamma'}\widetilde C_{\Gamma,j}(\gamma' f_1,\gamma' f_2)$ is an almost everywhere well defined function for $f_1,f_2\in C_0^\infty(\R)$ and $j=0,1,2$.  More precisely, for $f_1,f_2\in C_0^\infty(\R)$ the limit in \eqref{Riesz} converges for almost every $x\in\R$ and
\begin{align*}
M_{\gamma'}\widetilde C_{\Gamma,1}(\gamma' f_1,\gamma' f_2)(x)&=-\int_{\R^2}\widetilde F_{\Gamma}(x,y_1,y_2)f_1'(y_1) f_2(y_2)\gamma'(x)\gamma'(y_2)dy_1\,dy_2,\\
M_{\gamma'}\widetilde C_{\Gamma,2}(\gamma' f_1,\gamma' f_2)(x)&=-\int_{\R^2}\widetilde F_{\Gamma}(x,y_1,y_2)f_1(y_1) f_2'(y_2)\gamma'(x)\gamma'(y_1)dy_1\,dy_2.
\end{align*}
Furthermore $\widetilde C_{\Gamma,j}$ is continuous from $\gamma'C_0^1(\R)\times\gamma'C_0^1(\R)$ into $(\gamma'C_0^1(\R))'$, and for $ f_0, f_1, f_2\in C_0^\infty(\R)$,
\begin{align*}
&\<\widetilde C_{\Gamma,0}(\gamma' f_1,\gamma' f_2),\gamma' f_0\>=-\int_{\R^3}\widetilde F_{\Gamma}(x,y_1,y_2) f_0'(x) f_1(y_1) f_2(y_2)\gamma'(y_1)\gamma'(y_2)dx\,dy_1\,dy_2\\
&\<\widetilde C_{\Gamma,1}(\gamma' f_1,\gamma' f_2),\gamma' f_0\>=-\int_{\R^3}\widetilde F_{\Gamma}(x,y_1,y_2) f_0(x) f_1'(y_1) f_2(y_2)\gamma'(x)\gamma'(y_2)dx\,dy_1\,dy_2\\
&\<\widetilde C_{\Gamma,2}(\gamma' f_1,\gamma' f_2),\gamma' f_0\>=-\int_{\R^3}\widetilde F_{\Gamma}(x,y_1,y_2) f_0(x) f_1(y_1) f_2'(y_2)\gamma'(x)\gamma'(y_1)dx\,dy_1\,dy_2.
\end{align*}
\end{proposition}

\begin{proof}
Fix $f_1,f_2\in C_0^\infty$, and we start by showing that $M_{b_0}\widetilde C_{\Gamma,1}(\gamma' f_1,\gamma'f_2)$ can be realized as a bounded function.   Define for $\epsilon>0$ and $x\in\R$
\begin{align*}
C_\epsilon(x)=\int_{|x-y_1|,|x-y_2|>\epsilon}\widetilde K_{\Gamma,1}(x,y_1,y_2)f_1(y_1) f_2(y_2)\gamma'(x)\gamma'(y_1)\gamma'(y_2) dy_1\,dy_2.
\end{align*}
Note that $\partial_{y_1}\widetilde F_{\Gamma}(x,y_1,y_2)=\widetilde K_{\Gamma,1}(x,y_1,y_2)\gamma'(y_1)$, and we integrate by parts to rewrite $C_\epsilon$
\begin{align*}
C_\epsilon(x)&=\int_{|x-y_1|,|x-y_2|>\epsilon}\partial_{y_1}\widetilde F_{\Gamma}(x,y_1,y_2) f_1(y_1) f_2(y_2)\gamma'(x)\gamma'(y_2)dy_1\,dy_2\\
&=-\int_{|x-y_1|,|x-y_2|>\epsilon}\widetilde F_{\Gamma}(x,y_1,y_2) f_1'(y_1) f_2(y_2)\gamma'(x)\gamma'(y_2)dy_1\,dy_2\\
&\hspace{1cm}+\int_{|x-y_2|>\epsilon}\widetilde F_{\Gamma}(x,x-\epsilon,y_2) f_1(x-\epsilon) f_2(y_2)\gamma'(x)\gamma'(y_2)dy_2\\
&\hspace{1cm}-\int_{|x-y_2|>\epsilon}\widetilde F_{\Gamma}(x,x+\epsilon,y_2) f_1(x+\epsilon) f_2(y_2)\gamma'(x)\gamma'(y_2)dy_2\\
&=-\int_{|x-y_1|,|x-y_2|>\epsilon}\widetilde F_{\Gamma}(x,y_1,y_2) f_1'(y_1) f_2(y_2)\gamma'(x)\gamma'(y_2)dy_1\,dy_2\\
&\hspace{1cm}+\int_{|x-y_2|>\epsilon}\(\widetilde F_{\Gamma}(x,x-\epsilon,y_2)-\widetilde F_{\Gamma}(x,x+\epsilon,y_2)\) f_1(x-\epsilon) f_2(y_2)\gamma'(x)\gamma'(y_2)dy_2\\
&\hspace{1cm}+\int_{|x-y_2|>\epsilon}\widetilde F_{\Gamma}(x,x+\epsilon,y_2)( f_1(x-\epsilon)- f_1(x+\epsilon)) f_2(y_2)\gamma'(x)\gamma'(y_2)dy_2\\
&=I_\epsilon(x)+II_\epsilon(x)+III_\epsilon(x).
\end{align*}
We use that $f_1\in C_0^\infty(\R)$ to conclude that the when integrating by parts, the boundary terms at $y_1=\pm\infty$ vanish, leaving the $y_1=x\pm\epsilon$ above.  We now verify that the limits of $I_\epsilon(x)$, $II_\epsilon(x)$, and $III_\epsilon(x)$ each exist as $\epsilon\rightarrow0$.\\

\noindent\underline{$I_\epsilon$ converges:}  To compute this limit, we verify that the integrand of $I_\epsilon$ is an integrable function.  Note that $||L'||_{L^\infty}=\lambda<1$ implies 
\begin{align*}
|\widetilde F_\Gamma(x,y_1,y_2)|&\leq\frac{1}{|Re\((x-y_1)^2-(L(x)-L(y_1))^2+(x-y_2)^2-(L(x)-L(y_2))^2\)|^{1/2}}\\
&\less\frac{1}{(1-\lambda)^{1/2}}\frac{1}{|x-y_1|+|x-y_2|}.
\end{align*}
Now let $R_0>0$ be large enough so that $\supp( f_j)\subset B(0,R_0)$ for $j=1,2$, and it follows that

\begin{align*}
&\int_{\R^2}|\widetilde F_{\Gamma}(x,y_1,y_2) f_1'(y_1) f_2(y_2)\gamma'(x)\gamma'(y_2)|dy_1\,dy_2\\
&\hspace{.75cm}\leq\frac{||\gamma'||_{L^\infty}^2}{(1-\lambda)^{1/2}}\int_{|y_1|,|y_2|\leq R_0}\frac{|| f_1'||_{L^\infty}|| f_2||_{L^\infty}}{|x-y_1|+|x-y_2|}dy_1\,dy_2\less\frac{R_0}{(1-\lambda)^{1/2}}||f_1'||_{L^\infty}||f_2||_{L^\infty}.
\end{align*}
Therefore $I_\epsilon$ converges to an absolutely convergent integral as $\epsilon\rightarrow0$.\\

\noindent\underline{$II_\epsilon\rightarrow0$}:  First we make a change of variables in $II_\epsilon$ to rewrite 
\begin{align}
II_\epsilon=\int_{|y_2|>1}&h_\epsilon(x,y_2) f_1(x-\epsilon) f_2( x-\epsilon y_2)\gamma'( x-\epsilon y_2)dy_2,\label{IIe}\\
\text{ where }\;\;\;&h_\epsilon(x,y_2)=\epsilon\gamma'(x)\(\widetilde F_{\Gamma}(x,x-\epsilon, x-\epsilon y_2)-\widetilde F_{\Gamma}(x,x+\epsilon, x-\epsilon y_2)\).\notag
\end{align}
We wish to apply dominated convergence to $II_\epsilon$ as it is written in \eqref{IIe}.  First we show that the integrand converges to zero almost everywhere (in particular for every $y_2\neq 0$ and $x$ such that $\gamma'(x)$ exists):  For $y_2>0$ it follows that $ f_2(x-\epsilon y_2)\gamma'(x-\epsilon y_2)\rightarrow f_2(x)\gamma'(x-)$ as $\epsilon\rightarrow0^+$.  When $y_2<0$, it follows that $ f_2(x-\epsilon y_2)\gamma'(x-\epsilon y_2)\rightarrow f_2(x)\gamma'(x+)$ as $\epsilon\rightarrow0^+$.  So either way, the limit exists for $y_2\neq0$ and almost every $x$.  Now we show that $h_\epsilon(x,y_2)\rightarrow0$ for almost every $x,y_2\in\R$.  For any $x\in\R$ such that $\gamma'(x)$ exists and $y_2\neq0$, we compute
\begin{align*}
\lim_{\epsilon\rightarrow0}\epsilon\gamma'(x)\widetilde F_{\Gamma}(x,x-\epsilon, x-\epsilon y_2)&=\lim_{\epsilon\rightarrow0}\frac{\gamma'(x)}{\(\frac{(\gamma(x)-\gamma(x-\epsilon))^2}{\epsilon^2}+y_2^2\frac{(\gamma(x)-\gamma(x-\epsilon y_2))^2}{(\epsilon y_2)^2}\)^{1/2}}\\
&=\frac{\gamma'(x)}{\(\gamma'(x)^2+y_2^2\gamma'(x)^2\)^{1/2}}=\(1+y_2^2\)^{-1/2}.
\end{align*}
It follows that $\epsilon\gamma'(x)\widetilde F_{\Gamma}(x,x+\epsilon, x-\epsilon y_2)\rightarrow\(1+y_2^2\)^{-1/2}$ as well.  Therefore $h_\epsilon(x,y_2)\rightarrow0$ as $\epsilon\rightarrow0$ for all $x,y_2\in\R$ such that $\gamma'(x)$ exists and $y_2\neq0$.   Now in order to apply dominated convergence to \eqref{IIe}, we need only to show that $h_\epsilon(x,y_2)$ is integrable in $y_2$ independent of $\epsilon$.  Define $g_t=\gamma(x)-\gamma(x-t)$, which satisfies for all $s,t\in\R$
\begin{align*}
&|g_t|=|\gamma(x)-\gamma(x-t)|\leq||\gamma'||_{L^\infty}|t|\leq2|t|\\
&Re(g_s^2+g_t^2)=Re\[\(\gamma(x)-\gamma(x-s)\)^2\]+Re\[\(\gamma(x)-\gamma(x-t)\)^2\]\geq (1-\lambda)(s^2+t^2).
\end{align*}
Also it is easy to verify that if $\omega=re^{i\theta},\zeta=\rho e^{i\phi}\in\C$ both have positive real part, i.e. $\theta,\phi\in(-\pi/2,\pi/2)$, then
\begin{align*}
\left|\sqrt{\omega}+\sqrt{\zeta}\right|&\geq  \sqrt r\cos(\theta/2)+\sqrt\rho \cos(\phi/2)\\
&\geq \sqrt{ r\cos(\theta)}+\sqrt{\rho \cos(\phi)}=\sqrt{Re(\omega)}+\sqrt{Re(\zeta)}.
\end{align*}
Here we use that $\sqrt{\cos(\theta)}\leq\cos(\theta/2)$ for $\theta\in(-\pi/2,\pi/2)$.  Using these properties, we bound $h_\epsilon$
\begin{align*}
|h_\epsilon&(x,y_2)|=\epsilon\left|\frac{1}{\(g_\epsilon^2+g_{\epsilon y_2}^2\)^{1/2}}-\frac{1}{\(g_{-\epsilon}^2+g_{\epsilon y_2}^2\)^{1/2}}\right|\\
&=\epsilon\frac{\left|g_{-\epsilon}^2-g_\epsilon^2\right|}{\left|g_{-\epsilon}^2+g_{\epsilon y_2}^2\right|^{1/2}\left|g_\epsilon^2+g_{\epsilon y_2}^2\right|^{1/2}}\frac{1}{\left|\(g_{\epsilon}^2+g_{\epsilon y_2}^2\)^{1/2}+\(g_{-\epsilon}^2+g_{\epsilon y_2}^2\)^{1/2}\right|}\\
&\leq\epsilon\frac{|g_{-\epsilon}|^2+|g_\epsilon|^2}{Re\(g_{-\epsilon}^2+g_{\epsilon y_2}^2\)^{1/2}Re\(g_\epsilon^2+g_{\epsilon y_2}^2\)^{1/2}}\frac{1}{\[Re\(g_{\epsilon}^2+g_{\epsilon y_2}^2\)\]^{1/2}+Re\[\(g_{-\epsilon}^2+g_{\epsilon y_2}^2\)\]^{1/2}}\\
&\leq\frac{\epsilon}{(1-\lambda)}\frac{2(2\epsilon)^2}{\epsilon^2+(\epsilon y_2)^2}\frac{1}{2(1-\lambda)^{1/2}\(\epsilon^2+(\epsilon y_2)^2\)^{1/2}}\\
&\less\frac{1}{(1-\lambda)^{3/2}}\frac{1}{(1+|y_2|)^3}.
\end{align*}
Therefore $|h_\epsilon(x,y_2)|\less(1+|y_2|)^{-3}$, and we can apply dominated convergence to \eqref{IIe}.  Hence $II_\epsilon\rightarrow0$ as $\epsilon\rightarrow0$.\\

\noindent\underline{$III_\epsilon\rightarrow0$}:  For this term, we use the regularity and compact support of $ f_1$ to directly bound 
\begin{align*}
| III_\epsilon |&\less\int_{\epsilon<|x-y_2|<|x|+R_0+\epsilon}|\widetilde F_{\Gamma}(x,x+\epsilon,y_2)|\;| f_1(x-\epsilon)- f_1(x+\epsilon)| f_2(y_2)|dy_2\\
&\leq\frac{1}{(1-\lambda)^{1/2}}\int_{\epsilon<|x-y_2|<|x|+R_0+\epsilon}\frac{1}{|x-y_2|}(2\epsilon|| f_1'||_{L^\infty})|| f_2||_{L^\infty}dy_2\\
&\less\frac{||f_1'||_{L^\infty}||f_2||_{L^\infty}}{(1-\lambda)^{1/2}}\epsilon|\log(|x|+R_0+\epsilon)-\log(\epsilon)|.
\end{align*}
Recall we chose $R_0>0$ such that $\supp(f_j)\subset B(0,R_0)$ for $j=1,2$.  Hence $III_\epsilon\rightarrow0$ as $\epsilon\rightarrow0$, and so
\begin{align*}
\lim_{\epsilon\rightarrow0}C_\epsilon(x)=C(x)=-\int_{\R^2}\widetilde F_\Gamma(x,y_1,y_2) f_1'(y_1) f_2(y_2)\gamma'(x)\gamma'(y_2)dy_1\,dy_2,
\end{align*}
which is an absolutely convergent integral.  This verifies the absolutely convergent integral representation for $\widetilde C_{\Gamma,1}$ in Proposition \ref{p:welldefined1}.  It also follows from our estimate of $I_\epsilon$ that $C_\epsilon(x)$ is bounded uniformly in $x$; hence for $ f_0\in C_0^\infty(\R)$ and $\epsilon>0$
\begin{align*}
|C_\epsilon(x) f_0(x)\gamma'(x)|\less (1-\lambda)^{-1/2}||f_1'||_{L^\infty}||f_2||_{L^\infty}R_0| f_0(x)|,
\end{align*}
and by dominated convergence
\begin{align*}
-\int_{\R^3}\widetilde F_\Gamma(x,y_1,y_2)f_0(x) f_1'(y_1) f_2(y_2)\gamma'(x)\gamma'(y_2)dx\,dy_1\,dy_2&=\lim_{\epsilon\rightarrow0}\int_\R C_\epsilon(x)\gamma'(x) f_0(x)dx\\
&\hspace{-2.5cm}=\int_{\R}C(x) f_0(x)\gamma'(x)dx=\<\widetilde C_{\Gamma,1}(\gamma' f_1,\gamma' f_2),\gamma' f_0\>.
\end{align*}
Furthermore since the bounds of $I_\epsilon$, $II_\epsilon$, and $III_\epsilon$ are in terms of $||f_j||_{L^\infty}$, $||f_j'||_{L^\infty}$, and $R_0$ for $j=0,1,2$ it follows that $\widetilde C_{\Gamma,1}$ is continuous from $\gamma'C_0^1(\R)\times\gamma'C_0^1(\R)$ into $(\gamma'C_0^1(\R))'$.  By symmetry, the properties of $\widetilde C_{\Gamma,2}$ follow as well.  Also
\begin{align*}
\<\widetilde C_{\Gamma,0}(\gamma'f_1,\gamma'f_2),\gamma'f_0\>&\\
&\hspace{-3cm}=\int_\R\lim_{\epsilon\rightarrow0}\int_{|x-y_1|,|x-y_2|>\epsilon}\widetilde K_{\Gamma,0}(x,y_1,y_2)f_0(x)f_1(y_1)f_2(y_2)\gamma'(x)\gamma'(y_1)\gamma'(y_2)dy_1\,dy_2\,dx\\
&\hspace{-3cm}=\int_\R\lim_{\epsilon\rightarrow0}\int_{|x-y_1|,|x-y_2|>\epsilon}\partial_x\widetilde F_{\Gamma}(x,y_1,y_2)f_0(x)f_1(y_1)f_2(y_2)\gamma'(y_1)\gamma'(y_2)dy_1\,dy_2\,dx\\
&\hspace{-3cm}=\int_\R\lim_{\epsilon\rightarrow0}\int_{|x-y_1|,|x-y_2|>\epsilon}\partial_{y_1}\widetilde F_{\Gamma}(x,y_1,y_2)f_0(x)f_1(y_1)f_2(y_2)\gamma'(y_1)\gamma'(y_2)dy_1\,dy_2\,dx\\
&\hspace{-2.5cm}+\int_\R\lim_{\epsilon\rightarrow0}\int_{|x-y_1|,|x-y_2|>\epsilon}\partial_{y_2}\widetilde F_{\Gamma}(x,y_1,y_2)f_0(x)f_1(y_1)f_2(y_2)\gamma'(y_1)\gamma'(y_2)dy_1\,dy_2\,dx\\
&\hspace{-3cm}=\<\widetilde C_{\Gamma,1}(\gamma'f_1,\gamma'f_2),\gamma'f_0\>+\<\widetilde C_{\Gamma,2}(\gamma'f_1,\gamma'f_2),\gamma'f_0\>
\end{align*}
By the absolutely integrable representations of $\widetilde C_{\Gamma,1}$ and $\widetilde C_{\Gamma,2}$, it follows that
\begin{align*}
\<\widetilde C_{\Gamma,0}(\gamma'f_1,\gamma'f_2),\gamma'f_0\>&=\<\widetilde C_{\Gamma,1}(\gamma'f_1,\gamma'f_2),\gamma'f_0\>+\<\widetilde C_{\Gamma,2}(\gamma'f_1,\gamma'f_2),\gamma'f_0\>\\
&\hspace{-2cm}=-\int_{\R^3}\widetilde F_\Gamma(x,y_1,y_2)f_0(x)f_1'(y_1)f_2(y_2)\gamma'(x)\gamma'(y_2)dx\,dy_1\,dy_2\\
&\hspace{-1.5cm}-\int_{\R^3}\widetilde F_\Gamma(x,y_1,y_2)f_0(x)f_1(y_1)f_2'(y_2)\gamma'(x)\gamma'(y_1)dx\,dy_1\,dy_2
\end{align*}
\begin{align*}
&\hspace{0cm}=\lim_{\epsilon\rightarrow0}\int_{|x-y_1|,|x-y_2|>\epsilon}\partial_{y_1}\widetilde F_\Gamma(x,y_1,y_2)f_0(x)f_1(y_1)f_2(y_2)\gamma'(x)\gamma'(y_2)dx\,dy_1\,dy_2\\
&\hspace{.5cm}+\int_{|x-y_1|,|x-y_2|>\epsilon}\partial_{y_2}\widetilde F_\Gamma(x,y_1,y_2)f_0(x)f_1(y_1)f_2(y_2)\gamma'(x)\gamma'(y_1)dx\,dy_1\,dy_2\\
&\hspace{.5cm}+\int_{|x-y_2|>\epsilon}\widetilde F_\Gamma(x,x+\epsilon,y_2)f_0(x)f_1(x+\epsilon)f_2(y_2)\gamma'(x)\gamma'(y_2)dx\,dy_2\\
&\hspace{.5cm}-\int_{|x-y_2|>\epsilon}\widetilde F_\Gamma(x,x-\epsilon,y_2)f_0(x)f_1(x-\epsilon)f_2(y_2)\gamma'(x)\gamma'(y_2)dx\,dy_2\\
&\hspace{.5cm}+\int_{|x-y_1|>\epsilon}\widetilde F_\Gamma(x,y_1,x+\epsilon)f_0(x)f_1(y_1)f_2(x+\epsilon)\gamma'(x)\gamma'(y_1)dx\,dy_1\\
&\hspace{.5cm}-\int_{|x-y_1|>\epsilon}\widetilde F_\Gamma(x,y_1,x-\epsilon)f_0(x)f_1(y_1)f_2(x-\epsilon)\gamma'(x)\gamma'(y_1)dx\,dy_1\\
&\hspace{0cm}=\lim_{\epsilon\rightarrow0}\int_{|x-y_1|,|x-y_2|>\epsilon}\partial_x\widetilde F_\Gamma(x,y_1,y_2)f_0(x)f_1(y_1)f_2(y_2)\gamma'(y_1)\gamma'(y_2)dx\,dy_1\,dy_2.
\end{align*}
The boundary terms of the integration by parts here (the last 4 terms) tend to zero as $\epsilon\rightarrow0$ in the same way as they did for $II_\epsilon$ and $III_\epsilon$ above.  Now we will integrate by parts one more time in $x$ here to obtain an integral representation for $C_{\Gamma,0}$:
\begin{align*}
&\lim_{\epsilon\rightarrow0}\int_{|x-y_1|,|x-y_2|>\epsilon}\partial_x\widetilde F_\Gamma(x,y_1,y_2)f_0(x)f_1(y_1)f_2(y_2)\gamma'(y_1)\gamma'(y_2)dx\,dy_1\,dy_2\\
&\hspace{0cm}=\lim_{\epsilon\rightarrow0}-\int_{|x-y_1|,|x-y_2|>\epsilon}\widetilde F_\Gamma(x,y_1,y_2)f_0'(x)f_1(y_1)f_2(y_2)\gamma'(y_1)\gamma'(y_2)dx\,dy_1\,dy_2\\
&\hspace{.5cm}+\int_{|y_1-y_2-\epsilon|>\epsilon}\widetilde F_\Gamma(y_1-\epsilon,y_1,y_2)f_0(y_1-\epsilon)f_1(y_1)f_2(y_2)\gamma'(y_1)\gamma'(y_2)dy_1\,dy_2\\
&\hspace{.5cm}-\int_{|y_1-y_2+\epsilon|>\epsilon}\widetilde F_\Gamma(y_1+\epsilon,y_1,y_2)f_0(y_1+\epsilon)f_1(y_1)f_2(y_2)\gamma'(y_1)\gamma'(y_2)dy_1\,dy_2\\
&\hspace{.5cm}+\int_{|y_1-y_2+\epsilon|>\epsilon}\widetilde F_\Gamma(y_2-\epsilon,y_1,y_2)f_0(y_2-\epsilon)f_1(y_1)f_2(y_2)\gamma'(y_1)\gamma'(y_2)dy_1\,dy_2\\
&\hspace{.5cm}-\int_{|y_1-y_2-\epsilon|>\epsilon}\widetilde F_\Gamma(y_2+\epsilon,y_1,y_2)f_0(y_2+\epsilon)f_1(y_1)f_2(y_2)\gamma'(y_1)\gamma'(y_2)dy_1\,dy_2\\
&\hspace{0cm}=\lim_{\epsilon\rightarrow0}-\int_{|x-y_1|,|x-y_2|>\epsilon}\widetilde F_\Gamma(x,y_1,y_2)f_0'(x)f_1(y_1)f_2(y_2)\gamma'(y_1)\gamma'(y_2)dx\,dy_1\,dy_2\\
&\hspace{.5cm}+\int_{|y_1-y_2|>\epsilon}\widetilde F_\Gamma(y_1,y_1+\epsilon,y_2)f_0(y_1)f_1(y_1+\epsilon)f_2(y_2)\gamma'(y_1)\gamma'(y_2)dy_1\,dy_2\\
&\hspace{.5cm}-\int_{|y_1-y_2|>\epsilon}\widetilde F_\Gamma(y_1,y_1-\epsilon,y_2)f_0(y_1)f_1(y_1-\epsilon)f_2(y_2)\gamma'(y_1)\gamma'(y_2)dy_1\,dy_2\\
&\hspace{.5cm}+\int_{|y_1-y_2|>\epsilon}\widetilde F_\Gamma(y_2,y_1,y_2+\epsilon)f_0(y_2)f_1(y_1)f_2(y_2+\epsilon)\gamma'(y_1)\gamma'(y_2)dy_1\,dy_2\\
&\hspace{.5cm}-\int_{|y_1-y_2|>\epsilon}\widetilde F_\Gamma(y_2,y_1,y_2-\epsilon)f_0(y_2)f_1(y_1)f_2(y_2-\epsilon)\gamma'(y_1)\gamma'(y_2)dy_1\,dy_2\\
&=-\int_{\R^3}\widetilde F_\Gamma(x,y_1,y_2)f_0'(x)f_1(y_1)f_2(y_2)\gamma'(y_1)\gamma'(y_2)dx\,dy_1\,dy_2.
\end{align*}
Once again we use the same argument for the $II_\epsilon$ and $III_\epsilon$ terms to verify that these boundary terms (the last 4 terms) tend to zero as $\epsilon\rightarrow0$.  Then the pairing identity for $\widetilde C_{\Gamma,0}$ holds as well.  This completes the proof of Proposition \ref{p:welldefined1}.
\end{proof}

In the next proposition we extend $\widetilde C_\Gamma$ and $C_\Gamma$ to product Lebesgue spaces.

\begin{proposition}\label{p:application1}
Let $L$ be a Lipschitz function with Lipschitz constant $\lambda<1$ such that for almost every $x\in\R$ the limits
\begin{align*}
&\lim_{\epsilon\rightarrow0^+}\gamma'(x+\epsilon)=\gamma'(x+)& &\text{and}& &\lim_{\epsilon\rightarrow0^+}\gamma'(x-\epsilon)=\gamma'(x-)&
\end{align*}
exist.  If $L$ id differentiable off of some compact set and there exists $c_0\in\R$ such that 
\begin{align*}
\lim_{|x|\rightarrow\infty}L'(x)=c_0,
\end{align*}
then $\widetilde C_{\Gamma,j}$ is bounded $L^{p_1}(\R)\times L^{p_2}(\R)$ into $L^p(\R)$ for all $1<p_1,p_2<\infty$ satisfying \eqref{Holder} for each $j=0,1,2$.  Furthermore, $ C_{\Gamma,j}$ is bounded $L^{p_1}(\Gamma)\times L^{p_2}(\Gamma)$ into $L^p(\Gamma)$ for all $1<p_1,p_2<\infty$ satisfying \eqref{Holder} for each $j=0,1,2$.
\end{proposition}

\begin{proof}
We will apply Theorem \ref{t:fullTb} to $\widetilde C_{\gamma,1}$ with $b_0=b_1=b_2=\gamma'$.  Note that $\gamma'$ is para-accretive since $Re(\gamma')=1$ and $\gamma'\in L^\infty$.  It is not hard to see that $\widetilde K_{\Gamma,1}$ is the kernel function associated to $\widetilde C_{\Gamma,1}$.  It also follows from $||L'||_{L^\infty}=\lambda<1$ that $\widetilde K_{\Gamma,1}$ is a standard bilinear kernel:
\begin{align*}
|\widetilde K_{\Gamma,1}(x,y_1,y_2)|&\leq\frac{1}{(1-\lambda)^{3/2}}\frac{|\gamma(x)-\gamma(y_1)|}{|(x-y_1)^2+(x-y_2)^2|^{3/2}}\leq\frac{1}{(1-\lambda)^{3/2}}\frac{||\gamma'||_{L^\infty}}{(x-y_1)^2+(x-y_2)^2},
\end{align*}
and
\begin{align*}
\partial_{y_2}\widetilde K_{\Gamma,1}(x,y_1,y_2)&\leq\frac{3}{(1-\lambda)^{5/2}}\frac{|\gamma(x)-\gamma(y_1)|\,|\gamma(x)-\gamma(y_2)|\,|\gamma'(y_2)|}{((x-y_1)^2+(x-y_2)^2)^{5/2}}\\
&\leq\frac{3}{(1-\lambda)^{5/2}}\frac{||\gamma'||_{L^\infty}^3}{((x-y_1)^2+(x-y_2)^2)^{3/2}}.
\end{align*}
A simliar estimate holds for $\partial_{x}\widetilde K_{\Gamma,1}(x,y_1,y_2)$ and $\partial_{y_1}\widetilde K_{\Gamma,1}(x,y_1,y_2)$, which implies that $\widetilde K_{\Gamma,1}(x,y_1,y_2)$ is a standard bilinear kernel.  Now it remains to verify that $\widetilde C_{\Gamma,1}$ satisfies the WBP and the BMO testing conditions for $b_0=b_1=b_2=\gamma'$.  Let $\phi_0,\phi_1,\phi_2$ be normalized bumps of order $1$, $u\in\R$, and $R>0$.  By Proposition \ref{p:welldefined1}, we have
\begin{align*}
&\left|\<\widetilde C_{\Gamma,1}(\gamma'\phi_1^{u,R},\gamma'\phi_2^{u,R}),\gamma'\phi_0^{u,R}\>\right|\\
&\hspace{1cm}\leq\int_{\R^3}|\widetilde F_{\Gamma}(x,y_1,y_2)\phi_0^{u,R}(x)(\phi_1^{u,R})'(y_1)\phi_2^{u,R}(y_2)\gamma'(x)\gamma'(y_2)|dx\,dy_1\,dy_2\\
&\hspace{1cm}\less\frac{1}{R(1-\lambda)^{1/2}}\int_{[R,R]^3}\frac{dx\,dy_1\,dy_2}{|x-y_1|+|x-y_2|}\less \frac{R}{(1-\lambda)^{1/2}}
\end{align*}
So $\widetilde C_{\Gamma,1}$ satisfies the WBP.  Now we check the three $BMO$ conditions of Theorem \ref{t:fullTb}:\\

\noindent\underline{$M_{\gamma'}\widetilde C_{\Gamma,1}(\gamma',\gamma')=0\in BMO$}\,:  Let $\phi\in C_0^\infty(\R)$ such that $\gamma'\phi$ has mean zero and $\eta\in C_0^\infty(\R)$ such that $0\leq\eta\leq1$, $\eta=1$ on $[-1,1]$, $\supp(\eta)\subset[-2,2]$, and $\eta_R(x)=\eta(x/R)$.  Again we use Proposition \ref{p:welldefined1} and make a change of variables,
\begin{align}
\<\widetilde C_{\Gamma,1}(\gamma'\eta_R,\gamma'\eta_R),\gamma'\phi\>&=-\int_{\R^3}\widetilde F_{\Gamma}(x,y_1,y_2)(\eta_R)'(y_1)\gamma'(y_2)\eta_R(y_2)\gamma'(x)\phi(x)dy_1\,dy_2\,dx\notag\\
&=-\int_{\R^3}R\widetilde F_{\Gamma}(x,Ry_1,Ry_2)\phi(x)\eta'(y_1)\eta(y_2)\gamma'(x)\gamma'(Ry_2)dx\,dy_1dy_2.\label{C(1)}
\end{align}
Then for $y_1,y_2\neq0$, we can compute the pointwise limit
\begin{align*}
\lim_{R\rightarrow\infty}R\widetilde F_{\Gamma}(x,Ry_1,Ry_2)\gamma'(Ry_2)&=\lim_{R\rightarrow\infty}\frac{R\gamma'(Ry_2)}{((\gamma(x)-\gamma(Ry_1))^2+(\gamma(x)-\gamma(Ry_2))^2)^{1/2}}\\
&=\lim_{R\rightarrow\infty}\frac{\gamma'(Ry_2)}{\(y_1^2\frac{(\gamma(x)-\gamma(Ry_1))^2}{(Ry_1)^2}+y_2^2\frac{(\gamma(x)-\gamma(Ry_2))^2}{(Ry_2)^2}\)^{1/2}}\\
&=\frac{1+ic_0}{\(y_1^2(1+ic_0)^2+y_2^2(1+ic_0)^2\)^{1/2}}=\frac{1}{\(y_1^2+y_2^2\)^{1/2}}.
\end{align*}
Here we use that $L$ is differentiable off of a compact set, that $L'(x)\rightarrow c_0$ as $|x|\rightarrow\infty$, and L'Hospital's rule to conclude that $L(x)/x\rightarrow c_0$ as $|x|\rightarrow\infty$.  Now let $R>0$ be large enough so that $\supp(\phi)\subset B(0,R/4)$, and using that $\supp(\eta')\subset[-2,2]\backslash[-1,1]$, we have the estimate
\begin{align*}
|R\widetilde F_{\Gamma}(x,Ry_1,Ry_2)\gamma'(x)\gamma'(Ry_2)\phi(x)\eta'(y_1)\eta(y_2)|&\less\frac{1}{(1-\lambda)^{1/2}}\frac{R|\phi(x)\eta'(y_1)|}{|x-Ry_1|+|x-Ry_2|}\\
&\hspace{-3.75cm}\leq\frac{1}{(1-\lambda)^{1/2}}\frac{R|\phi(x)\eta'(y_1)|}{R|y_1|/2+R/2-2|x|+R|y_2|}\less\frac{1}{(1-\lambda)^{1/2}}\frac{1}{|y_1|+|y_2|}.
\end{align*}
Then by dominated convergence
\begin{align*}
\lim_{R\rightarrow\infty}\<\widetilde C_{\Gamma,1}(\gamma'\eta_R,\gamma'\eta_R),\gamma'\phi\>&=\int_\R\(-\int_{\R^2}\frac{\eta'(y_1)\eta(y_2)}{\(y_1^2+y_2^2\)^{1/2}}dy_1dy_2\)\phi(x)\gamma'(x)dx=0.
\end{align*}
Here we also use that $\gamma'\phi$ has mean zero.  Therefore $M_{\gamma'}\widetilde C_{\Gamma,1}(\gamma',\gamma')=0\in BMO$ in the sense of Definition \ref{d:Tb}.\\

\noindent\underline{$M_{\gamma'}\widetilde C_\Gamma^{*1}(\gamma',\gamma')=0\in BMO$}\,:  Note that for every $x,y_1,y_2\in\R$ such that $|x-y_1|+|x-y_2|\neq0$ we can write
\begin{align*}
\widetilde K_{\Gamma,1}(x,y_1,y_2)\gamma'(x)\gamma'(y_1)\gamma'(y_2)&=\partial_{y_1}\widetilde F_{\Gamma,1}(x,y_1,y_2)\gamma'(x)\gamma'(y_2)\\
&=\partial_{x}\widetilde F_{\Gamma,1}(x,y_1,y_2)\gamma'(y_1)\gamma'(y_2)-\partial_{y_2}\widetilde F_{\Gamma,1}(x,y_1,y_2)\gamma'(x)\gamma'(y_1)\\
&=(\widetilde K_{\Gamma,0}(x,y_1,y_2)-\widetilde K_{\Gamma,2}(x,y_1,y_2))\gamma'(x)\gamma'(y_1)\gamma'(y_2).
\end{align*}
Then it follows that $M_{\gamma'}\widetilde C_{\Gamma,1}(M_{\gamma'}\,\cdot,M_{\gamma'}\,\cdot)=M_{\gamma'}\widetilde C_{\Gamma,0}(M_{\gamma'}\,\cdot,M_{\gamma'}\,\cdot)-M_{\gamma'}\widetilde C_{\Gamma,2}(M_{\gamma'}\,\cdot,M_{\gamma'}\,\cdot)$, and so by Proposition \ref{p:welldefined1}
\begin{align*}
&\<\widetilde C_{\Gamma,1}^{*1}(\gamma'\eta_R,\gamma'\eta_R),\gamma'\phi\>=\<\widetilde C_{\Gamma,1}(\gamma'\phi,\gamma'\eta_R),\gamma'\eta_R\>\\
&\hspace{3cm}=\<\widetilde C_{\Gamma,0}(\gamma'\phi,\gamma'\eta_R),\gamma'\eta_R\>-\<\widetilde C_{\Gamma,2}(\gamma'\phi,\gamma'\eta_R),\gamma'\eta_R\>\\
&\hspace{3cm}=-\int_{\R^3}\widetilde F_{\Gamma}(x,y_1,y_2)\gamma'(y_1)\phi(y_1)\gamma'(y_2)\eta_R(y_2)(\eta_R)'(x)dy_1\,dy_2\,dx\\
&\hspace{3.5cm}+\int_{\R^3}\widetilde F_{\Gamma}(x,y_1,y_2)\gamma'(y_1)\phi(y_1)(\eta_R)'(y_2)\gamma'(x)\eta_R(x)dy_1\,dy_2\,dx.
\end{align*}
These two expressions tend to zero by the same argument that \eqref{C(1)} tends to zero as $R\rightarrow\infty$ in the proof of the $M_{\gamma'}\widetilde C_{\Gamma,1}(\gamma',\gamma')=0$ condition.  Therefore $M_{\gamma'}\widetilde C_\Gamma^{*1}(\gamma',\gamma')=0\in BMO$ as well.\\  

\noindent\underline{$M_{\gamma'}\widetilde C_\Gamma^{*2}(\gamma',\gamma')=0\in BMO$}\,:  By Proposition \ref{p:welldefined1}, we can compute
\begin{align*}
&\<\widetilde C_{\Gamma,1}^{*2}(\gamma'\eta_R,\gamma'\eta_R),\gamma'\phi\>=-\int_{\R^3}\widetilde F_{\Gamma}(x,y_1,y_2)(\eta_R)'(y_1)\gamma'(x)\eta_R(x)\gamma'(y_2)\phi(y_2)dy_1\,dy_2\,dx.
\end{align*}
Again, this expression is essentially the same as the one in \eqref{C(1)}, and hence tends to zero as $R\rightarrow\infty$ by the argument.  Therefore $M_{\gamma'}\widetilde C_\Gamma^{*2}(\gamma',\gamma')=0\in BMO$.\\

Then by Theorem \ref{t:fullTb}, $\widetilde C_{\Gamma,1}$ can be extended to a bounded operator from $L^{p_1}\times L^{p_2}$ into $L^p$ for appropriate $p,p_1,p_2$.  Now it is easy to prove that $C_{\Gamma,1}$ can also be extended to a bounded operator:  Let $1<p_1,p_2<\infty$ and $1/2<p<\infty$ satisfy \eqref{Holder}.  For $g_1\in L^{p_1}(\Gamma)$ and $g_2\in L^{p_2}(\Gamma)$, and it follows that
\begin{align*}
||C_{\Gamma,1}(g_1,g_2)||_{L^p(\Gamma)}&=\(\int_{\R}|C_{\Gamma,1}(g_1,g_2)(\gamma(x))|^p|\gamma'(x)|dx\)^{1/p}\\
&\leq||\gamma'||_{L^\infty}^{1/p}||\widetilde C_{\Gamma,1}||_{p,p_1,p_2}||g_1\circ\gamma^{-1}||_{L^{p_1}(\R)}||g_2\circ\gamma^{-1}||_{L^{p_2}(\R)}\\
&=||\gamma'||_{L^\infty}^{1/p}||\gamma'^{-1}||_{L^\infty}^{1/p}||\widetilde C_{\Gamma,1}||_{p,p_1,p_2}||g_1||_{L^{p_1}(\Gamma)}||g_2||_{L^{p_2}(\Gamma)}.
\end{align*}
The bounds for $\widetilde C_{\Gamma,0}$, $\widetilde C_{\Gamma,2}$, $C_{\Gamma,0}$, and $C_{\Gamma,2}$ follow in the same way.
\end{proof}

\bibliographystyle{amsplain}

\end{document}